\documentclass[english,11pt]{article}%
\usepackage[T1]{fontenc}
\usepackage[latin9]{inputenc}
\usepackage{amsmath}
\usepackage{amssymb}
\usepackage{amsfonts}
\usepackage{enumerate}
\usepackage{babel}
\usepackage{enumerate}
\usepackage{fancyhdr}
\usepackage{color}
\usepackage[toc]{appendix}
\usepackage{graphicx}%
\usepackage[ruled,vlined]{algorithm2e}
\providecommand{\U}[1]{\protect\rule{.1in}{.1in}}
\makeatletter
\newtheorem{theorem}{Theorem}

\newtheorem{corollary}[theorem]{Corollary}

\newtheorem{lemma}[theorem]{Lemma}

\newtheorem{proposition}[theorem]{Proposition}
\newtheorem{remark}[theorem]{Remark}

\numberwithin{equation}{section}
\numberwithin{example}{section}
\numberwithin{theorem}{section}

\newenvironment{proof}[1][Proof]{\noindent\textbf{#1.} }{\ \rule{0.5em}{0.5em}}
\makeatother
\begin{document}

\title{Almost sure convergence of the accelerated weight histogram algorithm}
\author{Henrik Hult and Guo-Jhen Wu}
\maketitle

\begin{abstract}
The accelerated weight histogram (AWH) algorithm is an iterative extended ensemble algorithm, developed for statistical physics and computational biology applications. It is used to estimate free energy differences and expectations with respect to Gibbs measures. The AWH algorithm is based on iterative updates of a design parameter, which is closely related to the free energy,  obtained by matching a weight histogram with a specified target distribution. The weight histogram is constructed from samples of a Markov chain on the product of the state space and parameter space. In this paper almost sure convergence of the AWH algorithm is proved, for estimating free energy differences as well as estimating expectations with adaptive ergodic averages. The proof is based on identifying the AWH algorithm as a stochastic approximation and studying the properties of the associated limit ordinary differential equation. 
\end{abstract}

\section{Introduction}

Consider a parameterized family of Gibbs measures which have densities $\{p_{X | \Lambda}( \cdot | \lambda), \lambda \in \mathcal{L}\}$ with respect to a common reference measure $\nu$ on a state space $\mathcal{X}$ given by
\begin{align}\label{Gibbs}
	p_{X | \Lambda}(x | \lambda) = \exp\{ - E(x, \lambda) + 	F(\lambda)\}, 
\end{align}
where $E(x, \lambda)$ is the energy, $F(\lambda)$  $= - \log \int \exp\{-E(x, \lambda)\} \nu(dx)$ is the free energy (i.e., $e^{-F(\lambda)}$ is the normalizing constant) and $\mathcal{L}$ is the parameter space. Such models originate from statistical physics, but are frequently encountered in various fields including Bayesian statistics, artificial intelligence and machine learning, computational biology, etc. The interests lie in computing free energy differences, $F(\lambda) - F(\sigma)$ with $\lambda,\sigma\in\mathcal{L}$, and
expectations with respect to $p_{X | \Lambda}(\cdot | \lambda)$ at a fixed parameter $\lambda \in \mathcal{L}$. In classical models of statistical physics, $\mathcal{X}$ represents the state space of a system of particles and the energy is commonly of the form $E(x, \lambda) = \lambda E(x)$, with $\lambda$ as the inverse temperature. More generally, the parameter $\lambda$ may represent a number of reaction coordinates. In statistical machine learning models, the parameter space $\mathcal{L}$ may be high-dimensional, consisting of millions of parameters. 

Sampling configurations according to a Gibbs measure is challenging because the free energy is unknown; the state space is often too large, which makes summation/integration over the state space infeasible. A widely used approach is Markov chain Monte Carlo (MCMC), where an ergodic Markov chain is constructed with $p_{X| \Lambda}(\cdot | \lambda)$ as its invariant distribution, see \cite{GRS96, AdFDI03, RC04, AG07}. The vanilla MCMC algorithms often suffer from slow mixing (slow relaxation) that reduces the effective number of samples and even may lead to biased results. In physical problems there are several potential reasons that lead to slow mixing, including critical slow down near points of second order phase transitions, nucleation associated with first order phase transitions and trapping in metastable states around local minima in complex energy landscapes. The latter problem is often encountered in  models of spin glasses \cite{SW86, BC92}, interacting spins in random fields \cite{MP92}, and heteropolymers \cite{UT96, HO99} as well as in complex statistical inference problems where models are highly non-Gaussian, see e.g.\ \cite{GRS96, RC04, RR04}.

Extended ensemble methods refer to computational methods where multiple copies of the state and parameter spaces are considered and a Markov chain on the extended space is constructed. Such methods have been developed since the mid 80's to overcome the problems of slow mixing due to nucleation and metastability. They include well known algorithms such as parallel tempering/replica exchange, see \cite{SW86}, and related exchange Monte Carlo methods, simulated tempering, see \cite{MP92}, multicanonical Monte Carlo (adaptive umbrella sampling), see \cite{M87, HvEK92, BN91, BN92, B00}, as well as the more recent infinite swapping algorithm, see \cite{DLP12},  and the accelerated weight histogram (AWH) algorithm introduced by Lidmar in \cite{LLH14}. More recent advances also include generative models from machine learning that can be trained to approximate samples from Gibbs distributions, see \cite{NOKW19}. 
 
The main idea of extended ensemble algorithms is to sample from artificial ensembles that are constructed as extensions or compositions
of the original ensemble. Fast mixing of the Markov chain in higher temperature (energy, etc.) components of the artificial ensembles greatly facilitate the mixing in other components. Averages over the original ensemble are calculated by marginalization (e.g.~parallel tempering), conditional
sampling (e.g.~simulated tempering), or, a reweighting procedure (e.g.~multicanonical
Monte Carlo, infinite swapping).  Improved mixing is not the only benefit of extended ensemble methods. They are particularly useful for estimating free energy differences $F(\lambda) - F(\sigma)$ for $\lambda, \sigma \in \mathcal{L}$ and  expectations $\int \psi(x) p_{X | \Lambda}(x | \lambda) \nu(dx)$ for a given function $\psi$ and parameter $\lambda \in \mathcal{L}$.  They are also useful for efficient sampling of rare events, and may be used to estimate the probability of configurations with small probability or transition probabilities between different metastable states.

The design of extended ensemble methods is intricate. There is a lot of flexibility in selecting the extended variables, the target density on the extended space, the transition kernel of the Markov chain as well as particular design parameters such as temperature selection and spacings in the parameter space. The current practice is to a large extent based on heuristics and experimental success. Methods based on the spectral gap to study convergence of ergodic Markov chains, as well as analysis of the asymptotic variance of particular observables, are generally insufficient to fully capture the effect of design parameters. Recent advances in the theory of large deviations for empirical measures of Markov chains shows promising results for studying the rate of convergence of weighted empirical averages towards population averages in the context of parallel tempering and infinite swapping, see \cite{DW20,DW21}.

Although large deviations theory may provide useful information for algorithm design, it is very challenging to make detailed calculations so that they provide optimal designs in complex models. For example, the problem of optimal temperature selection in an infinite swapping algorithm for the two-well potential has recently been worked out by Dupuis and Wu, see \cite{DW21}, and even such relatively simple model requires extensive and ingenious analysis. For a general multi-well potential, in order to make the calculations and analysis feasible, they have to comprise with a sub-optimal (but nearly optimal) design. It is our belief that including a learning mechanism is necessary in complex models, where design parameters may be updated adaptively during the simulation. 

An example of this learning mechanism is provided in the accelerated weight histogram method, see \cite{LLH14}, where an estimate of the free energy is updated iteratively during the simulation and the invariant density of the extended Markov chain is updated accordingly. Ultimately, this iterative procedure results in a fixed point method that can be analyzed using the theory of stochastic approximation (Robbins-Monro).  

The AWH algorithm has connections to the Wang-Landau algorithm, see \cite{WL01,BP07}, which can be viewed as an adaptive importance sampling algorithm belonging to the class of free energy biasing techniques, see e.g.\ \cite{LRS07}. The Wang-Landau method was originally developed for simulations in the closely related multicanonical ensemble, see \cite{ZM07}, but is straightforward to adapt to extended ensemble simulations, see \cite{BN92}. In the Wang-Landau algorithm the biasing factor is updated during the simulation to flatten the target probability, but the design is delicate since the adaption mechanism must decrease slowly with the number of simulations, as first observed in \cite{BP07}. By representing the Wang-Landau algorithm as a stochastic approximation, the almost sure convergence of the Wang-Landau algorithm is proved in \cite{FJKLS15}, using the general conditions provided in \cite{AMP05}. 

In contrast to the Wang-Landau algorithm, the AWH algorithm allows for large parameter steps by the use of a Gibbs sampler combined with a reweighting procedure that takes advantage of the information collected during the simulation. This allows for a rather densely spaced set of parameters, without being limited by slow diffusion. Moreover, the free energy parameters are updated based on a histogram of weights, rather than a histogram of visited parameter values, which is combined with the information collected during previous iterations. 

In this paper the almost sure convergence of the AWH algorithm is proved by identifying it as a stochastic approximation algorithm. Free energy differences are considered jointly with ergodic averages for fixed parameters. The stability issue is circumvented by assuming that the iterates take values in a compact set and almost sure convergence is proved by verifying general conditions for stochastic approximations with state-dependent noise in \cite[Chapter 6]{KY03}. The main technical difficulties lie in the identification of the limit ordinary differential equation (ODE) for the stochastic approximation algorithm, and the construction of appropriate Lyapunov functions for the limit ODE to characterize its limit set. Extensions to non-projected algorithms may also be treated, using the conditions in \cite{AMP05}.  

The paper is organized as follows. In Section \ref{sec:AWH} the AWH algorithm is presented in detail, explaining its use for estimating free energy differences and ergodic averages.  Section \ref{sec:convergence} contains the results on almost sure convergence of estimates of free energy differences and Section \ref{sec:erg}  extends the results to include the joint estimation of free energy differences and ergodic averages at fixed parameters.   The Appendix contains existing results on almost sure convergence of stochastic approximations that are used to prove the main results.

\section{The Accelerated Weight Histogram Algorithm} \label{sec:AWH}

In this section the AWH algorithm is presented in detail. It was first introduced in \cite{Lidmar12} to estimate the free energy differences and further developed for applications in computational biology in \cite{LLH14}. The design of an appropriate target distribution for the AWH algorithm is analyzed in \cite{LLH18}. 

Let $\mathcal{X}$ be the configuration space and $\mathcal{L}$ be
the parameter space. For simplicity, it is assumed that $\mathcal{{X}}$
and $\mathcal{L}$ are both finite sets. We are concerned with the setting where $|\mathcal{X}|$ is much larger than $|\mathcal{L}|$ and summation over $\mathcal{X}$ is infeasible, whereas summation over $\mathcal{L}$ is feasible. Consider a family of Gibbs measures $\{p_{X | \Lambda}( \cdot | \lambda), \lambda \in \mathcal{L}\}$ on $\mathcal{X}$, where
\begin{align*}
	p_{X | \Lambda}(x | \lambda) \doteq \exp\{ - E(x, \lambda) + 	F(\lambda)\}, 
\end{align*}
$E(x, \lambda)$ is the energy and $F(\lambda) = - \log \sum_{x \in \mathcal{X}} \exp\{-E(x, \lambda)\}$ is the unknown free energy.  In the accelerated weight histogram algorithm a joint distribution on the extended state space $\mathcal{X} \times \mathcal{L}$ is constructed as
\begin{align}\label{eq:pjoint}
	p_{X,\Lambda |\Theta}(x,\lambda | \theta) \propto \exp\{ -E(x, \lambda) + \theta(\lambda)\},
\end{align}
where $\theta: \mathcal{L} \to \mathbb{R}$ represents a design parameter, which ideally is equal to $F$ up to an additive constant. The joint distribution is obtained by specifying the marginal distribution on $\mathcal{L}$ as
\begin{align}\label{eq:plambda}
    	p_{\Lambda | \Theta}(\lambda | \theta) =  \frac{e^{\theta(\lambda)-F(\lambda)}}{\sum_{\sigma \in \mathcal{L}}  e^{\theta(\sigma)-F(\sigma)} }. 
\end{align}
Consequently, by Bayes' theorem, the conditional distribution of $\Lambda$ given $X$ and $\Theta$ is formulated by
\begin{align*}
	p_{\Lambda | X, \Theta}(\lambda | x, \theta) = \frac{e^{-E(x, \lambda) + \theta(\lambda)}}{\sum_{\sigma \in \mathcal{L}} e^{-E(x,\sigma)+\theta(\sigma)}}.
\end{align*}
Note that, by construction $p_{X|\Lambda, \Theta}(x | \lambda, \theta) = p_{X| \Lambda}(x|\lambda)$ since
\[
 p_{X|\Lambda , \Theta}(x|\lambda, \theta) = \frac{p_{X, \Lambda | \Theta}(x | \lambda, \theta)}{p_{\Lambda | \Theta}(\lambda | \theta)} = \exp\{-E(x,\lambda) + F(\lambda)\} = p_{X| \Lambda}(x|\lambda).
 \]
Additionally, throughout the paper, $\theta$ is called the design parameter, which can be thought of as an approximation of the free energy, not to be confused with $\lambda$, which is a parameter that of the energy function and denotes a point in the parameter space $\mathcal{L}$. 

Let $\rho$ be a probability measure on $\mathcal{L}$, called the target distribution, which is specified by the user. Throughout the paper, we only consider $\rho$ which satisfies $\rho(\lambda)>0$ for all $\lambda\in\mathcal{L}$. The aim of the AWH algorithm is to find the design parameter $\theta_*$ that makes 
\begin{align*}
  p_{\Lambda | \Theta}(\lambda | \theta_*) = \rho(\lambda), \quad \lambda \in \mathcal{L}. 
\end{align*}
We call such $\theta_*$ the optimal design parameter. With such $\theta_*$, free energy differences may be computed by
\begin{align}\label{eqn:free_energy_diff}
    F(\lambda) - F(\sigma) = \theta_*(\lambda)-\theta_*(\sigma) - \log \rho(\lambda) + \log \rho(\sigma)
\end{align}
since 
\[
    \frac{\rho(\lambda)}{\rho(\sigma)} = \frac{p_{\Lambda | \Theta}(\lambda | \theta_*)}{p_{\Lambda | \Theta}(\sigma | \theta_*)}
    = \frac{e^{ \theta_*(\lambda)-F( \lambda)}}{e^{ \theta_*(\sigma)-F(\sigma) }}
\]
for any $\lambda,\sigma\in\mathcal{L}$. However,  computing $\theta_*$ is equally difficult as computing the free energy, $F$.

In the AWH algorithm estimates of free energy differences are obtained from estimates of $\theta_*$. An iterative scheme is introduced where an estimate $\theta_n$ of $\theta_*$ is updated in each iteration using weight histograms, as described in the following. Let $W_{n,k}: \mathcal{L} \to [0,\infty)$, $n=0,\dots, N-1$, $k = 1, \dots, N_I$ be weight histograms with $W_{0,0}(\lambda) = 0$, $\lambda \in \mathcal{L}$. 
The algorithm is initiated with an initial estimate $\theta_0$ of the design parameter. A Markov chain  $\{(X_{k}, \Lambda_{k}), k \geq 0\}$ on the extended state space $\mathcal{X}\times\mathcal{L}$ is constructed to have $p_{X, \Lambda | \theta_0}$ as its unique invariant distribution. The Markov chain is simulated for $N_I$ time steps to produce updated weight histograms on the parameter space $\Lambda$ as follows:
\begin{align*}
    W_{0,k}(\lambda) = W_{0,k-1}(\lambda) + p_{\Lambda | X, \Theta}(\lambda | X_{k}, \theta_0), \quad \lambda \in \mathcal{L}, \; k =1, \dots, N_I. 
\end{align*}
An updated estimate $\theta_1$ of the optimal design parameter $\theta_*$ is obtained by matching the normalized weight histogram $W_{0,N_I}/N_I$ with the target distribution $\rho$, 
\begin{align*}
    \theta_{1}(\lambda) = \theta_0(\lambda) - \log\left(\frac{W_{0,N_I}(\lambda)}{N_I \rho(\lambda)}\right), \quad \lambda \in \mathcal{L}.
\end{align*}
The procedure is iterated until a termination criteria is fulfilled; in this paper the algorithm terminates after $N$ updates of the estimate of the optimal design parameter. A pseudo-code for the AWH algorithm is provided in Algorithm \ref{alg:pseudo} and additional details provided below. 

\begin{algorithm}[H] \label{alg:pseudo}
\SetAlgoLined
\KwResult{$\theta_{N}$}
 Initialize: $n = 0$, $X_0$, $\Lambda_0$, $\theta_0$, $\rho$, $W_{0,0}=0$ \;
 \While{$n < N$}{
    \For{$k=1,\dots,N_I$}{
    Sample $X_{nN_I+k}$ from $	q(X_{nN_I+k-1}, x |  \Lambda_{n N_I + k-1})$ \; 
    Sample $\Lambda_{nN_I+k}$ from $p_{\Lambda  | X,\Theta}(\lambda | X_{nN_I+k}, \theta_n)$\;
    Update $W_{n,k}(\lambda) \leftarrow W_{n,k-1}(\lambda) + p_{\Lambda | X, \Theta}(\lambda | X_{nN_I+k}, \theta_n) 
    $, $\lambda \in \mathcal{L}$\;
 }
 Update $\theta_{n+1}(\lambda) \leftarrow \theta_n(\lambda) - \log\left(\frac{W_{n,N_I}(\lambda)}{(n+1) N_I \rho(\lambda)}\right)$, $\lambda \in \mathcal{L}$ \;
 Set $W_{n+1,0}(\lambda) \leftarrow (n+1) N_I \rho(\lambda)$, $\lambda \in \mathcal{L}$ \;
 Update $n \leftarrow n+1$\;
 }
 \caption{Accelerated weight histogram}
\end{algorithm}

In the $n$th iteration the sampling steps aim to sample the Markov chain $\{(X_{nN_I+k}, \Lambda_{nN_I+k}),k\geq 1\}$ with invariant distribution $p_{X,\Lambda | \Theta}(x,\lambda | \theta_n)$. The two sampling steps in Algorithm \ref{alg:pseudo} form a Gibbs sampler with the desired invariant distribution. Note that, in the first sampling step, one would ideally sample $X_{nN_I+k}$ from $p_{X | \Lambda}(x | \Lambda_{n N_I + k-1})$, but this is difficult since the free energy $F$ is unknown. Therefore the ideal step is replaced by sampling from a transition probability $q(X_{nN_I + k-1}, x | \Lambda_{nN_I+k-1})$ whose invariant distribution is $p_{X | \Lambda}(x | \Lambda_{n N_I + k-1})$.  That is,  $q( x,y| \lambda)$ satisfies
\begin{align}\label{eq:invariance}
    p_{X | \Lambda}(y | \lambda) = \sum_{x \in \mathcal{X}} q(x,y | \lambda) p_{X | \Lambda}(x | \lambda), \quad y \in \mathcal{X}, \lambda \in \mathcal{L}. 
\end{align}
Moreover, we assume that for any $\lambda\in\mathcal{L}$, 
\begin{align}
    \label{eq:irr}
    \{q(x,y | \lambda)\}_{x,y\in\mathcal{X}}, \; \text{is irreducible and aperiodic.}
\end{align}
This assumption will come in handy later on in the proof of almost sure convergence, and the assumption can be fulfilled easily in an implementation.

In contrast, since we only consider the case when $|\mathcal{L}|$ is not too large, the normalizing constant $\sum_{\sigma \in \mathcal{L}} e^{-E(x,\sigma) + \theta_n(\sigma)}$ can be explicitly computed in this case, which means that sampling $\Lambda_{nN_I+k}$ from $p_{\Lambda  | X,\Theta}(\lambda | X_{nN_I+k}, \theta_n)$ is straightforward.
\subsection{Estimation of free energy differences}
For $n\geq 1$, the weight histogram can be written as 
\begin{align*}
    W_{n,k}(\lambda) = W_{n,0}(\lambda) + \sum_{j=1}^k p_{\Lambda | X, \Theta}(\lambda | X_{nN_I+j}, \theta_n), \quad \lambda \in \mathcal{L}, k = 1\dots, N_I.
\end{align*}
For large $N_I$ the normalized weight histogram approximates a convex combination of the target density $\rho(\lambda)$ and the marginal density $p_{\Lambda | \Theta}(\lambda | \theta_n)$. Indeed, 
\begin{align*}
\frac{1}{(n+1)N_I}W_{n,N_I}(\lambda) &=\frac{n}{n+1}\rho(\lambda) + \frac{1}{(n+1) N_I}\sum_{k=1}^{N_I} p_{\Lambda | X, \Theta}(\lambda|X_{nN_I+k}, \theta_n)  \\ &
\rightarrow \frac{n}{n+1} \rho(\lambda) + \frac{1}{n+1}\sum_{x\in\mathcal{X}}p_{\Lambda | X, \Theta}(\lambda | x, \theta_n) p_{X | \Theta}(x | \theta_n) \\ &= \frac{n}{n+1}\rho(\lambda) + \frac{1}{n+1} p_{\Lambda | \Theta}(\lambda | \theta_n), \quad \lambda \in \mathcal{L},
\end{align*}
where the convergence is almost surely, as $N_I \rightarrow\infty$ , by the ergodic theorem. The updating scheme of the design parameter is then approximately
\begin{align*}
    \theta_{n+1}(\lambda) &= \theta_n(\lambda) -\log\left(\frac{W_{n,N_I}(\lambda)}{(n+1)N_I \rho(\lambda)}\right)
    \\
    &\approx \theta_n(\lambda) - \log\left(1-\frac{1}{n+1}\left(1-\frac{p_{\Lambda | \Theta}(\lambda | \theta_n)}{\rho(\lambda)}\right) \right), 
        \quad \lambda \in \mathcal{L},
\end{align*}
which has a fixed point at $\theta_*$ such that $p_{\Lambda | \Theta}(\lambda | \theta_*) = \rho(\lambda)$. Thus, $\theta_n$ computed by the AWH algorithm can be identified as a stochastic approximation of the optimal design parameter $\theta_*$. 

After termination of the AWH algorithm, e.g., after $N$ iterations, estimates of free energy differences $F(\lambda) - F(\sigma)$ may be computed as
\begin{align*}
    \theta_{N}(\lambda) - \theta_N(\sigma) - \log \rho(\lambda) + \log \rho(\sigma) ,\quad \lambda, \sigma \in \mathcal{L},
\end{align*}
since \eqref{eqn:free_energy_diff} and $\theta_N$ is an estimate of $\theta_*$. This is one of the main use of the AWH algorithm in \cite{Lidmar12}.
 
\subsection{Adaptive estimation of expectations with ergodic averages}\label{subsec:2.2}
In addition to estimating free energy differences, the AWH algorithm may also be used to estimate expectations at a fixed parameter value $\lambda \in \mathcal{L}$, 
\begin{align*}
    \sum_{x\in\mathcal{X}}\psi(x)p_{X | \Lambda}(x|\lambda),
\end{align*}
for some function $\psi:\mathcal{X}\rightarrow\mathbb{R}$. At the $n$th iteration, given design parameter $\theta_{n}$, the Markov chain $\{(X_{nN_I + k},\Lambda_{nN_I + k}), k=1, \dots, N_I\}$ has invariant distribution  $p_{X,\Lambda | \Theta}(x,\lambda | \theta_n)$ given by \eqref{eq:pjoint}. By the ergodic theorem, as $k \to \infty$,
\begin{align*}
\frac{1}{k}\sum_{j=1}^{k}\psi(X_{nN_I + j})p_{\Lambda | X, \Theta}(\lambda | X_{nN_I + j}, \theta_n)  &\rightarrow\sum_{x\in\mathcal{X}}\psi(x)p_{\Lambda | X, \Theta}(\lambda | x, \theta_n) p_{X | \Theta}(x | \theta_n) \\ &=\sum_{x\in\mathcal{X}}\psi(x)p_{X, \Lambda | \Theta}(x,\lambda | \theta_n),
\end{align*}
and 
\begin{align*}
\frac{1}{k}\sum_{j=1}^{k}p_{X | \Lambda, \Theta}(\lambda | X_{nN_I +j}, \theta_n) & \rightarrow\sum_{x\in\mathcal{X}}p_{\Lambda | X,\Theta}(\lambda | x, \theta_n)p_{X| \Theta}(x | \theta_n) = p_{\Lambda | \Theta}(\lambda | \theta_n).
\end{align*}
These two limits imply that, as $k\to \infty$, 
\begin{align*}
\frac{\sum_{j=1}^{k}\psi(X_{nN_I + j})p_{\Lambda | X, \Theta}(\lambda | X_{nN_I + j}, \theta_n)}{\sum_{j=1}^{k}p_{X | \Lambda, \Theta}(\lambda | X_{nN_I +j}, \theta_n)}  &\rightarrow \frac{\sum_{x\in\mathcal{X}}\psi(x)p_{X, \Lambda | \Theta}(x,\lambda | \theta_n)}{p_{\Lambda | \Theta}(\lambda | \theta_n)} \\ & = \sum_{x\in\mathcal{X}}\psi(x)p_{X | \Lambda}(x | \lambda).
\end{align*}
Thus, it is in principle possible to estimate the expectation on the right side in the last display by the left side in the last display for some large $k\in\mathbb{N}$. Such estimate is inefficient since it would require fixing $\theta_n$ and let the Markov chain run for a long time without updating $\theta_n$, not taking advantage of the possible improved estimates of the optimal design parameter along the simulation. A more efficient estimate is to keep updating $\theta_n$ according to the AWH algorithm and use the estimate 
\begin{align} \label{eq:ergest2}
\frac{\sum_{n=0}^{N-1}\sum_{j=1}^{N_I}\psi(X_{nN_I + j})p_{\Lambda | X, \Theta}(\lambda | X_{nN_I + j}, \theta_n)}{\sum_{n=0}^{N-1}\sum_{j=1}^{N_I}p_{X | \Lambda, \Theta}(\lambda | X_{nN_I +j}, \theta_n)}. 
\end{align}
As $\theta_n$ converges to the optimal $\theta_*$, it seems plausible that $\theta_n$ will essentially remain constant for large values of $n$ and the estimate in \eqref{eq:ergest2} will converge to $\sum_{x\in\mathcal{X}}\psi(x)p_{X | \Lambda}(x | \lambda)$ almost surely. This is indeed the case, as will be demonstrated in Section \ref{sec:erg}.  The proof is based on studying the convergence an extended stochastic approximation of $\theta_*$ and integrals of the form $\sum_{x\in\mathcal{X}} \psi(x)p_{X, \Lambda |\Theta}(x, \lambda | \theta_*)$.

\section{Almost Sure Convergence of Free Energy Differences}
\label{sec:convergence}

 In this section almost sure convergence of the AWH algorithm for estimating free energy differences will be proved. The proof is based on the ODE approach for proving almost sure convergence of projected stochastic approximation algorithms, as presented in \cite{KY03} and explained in more detail in the Appendix. In this approach, originally introduced by Ljung \cite{ljung77}, it is demonstrated that the noise averages out so that, asymptotically, the stochastic approximation algorithm may be interpreted as a discretization of a limit ODE. Consequently, the set of limit points of the stochastic approximation algorithm is a subset of the limit points of the limit ODE, which may be identified as a subset of the level set of $0$ of a properly constructed Lyapunov function.

 Our proof relies on identifying the AWH algorithm as a stochastic approximation algorithm, identification of the associated limit ODE, the construction of an associated Lyapunov function, and verifying the appropriate set of conditions for almost sure convergence. 
 
 Let us begin by identifying the AWH algorithm as a stochastic approximation.  Recall that the recursion in the AWH algorithm is given by
\begin{align*}
\theta_{n+1}(\lambda) & =\theta_{n}(\lambda)-\log\left(\frac{W_{nN_I}(\lambda)}{(n+1)N_I\rho(\lambda)}\right)\\
 & =\theta_{n}(\lambda)-\log\left(1-\frac{1}{n+1}+\frac{1}{n+1}\frac{\sum_{k=1}^{N_I} p_{\Lambda | X, \Theta}(\lambda | X_{nN_I+k}, \theta_n)}{N_I\rho(\lambda)}\right).
\end{align*}
Since $\log(1+x) = x + o(|x|)$ the recursion can be written as
\begin{align*}
& \theta_{n+1}(\lambda) = 
\theta_{n}(\lambda) +\frac{1}{n+1}\left(1- \frac{\sum_{k=1}^{N_I} p_{\Lambda | X, \Theta}(\lambda | X_{nN_I+k}, \theta_n)}{N_I\rho(\lambda)}+ \beta_n(\lambda) \right), 
\end{align*}
$\lambda \in \mathcal{L}$, where $\beta_n(\lambda) = o(1)$ almost surely. To guarantee realistic estimates it is practical to project the updates on a hyper-rectangle $H$ containing the origin. The projected algorithm circumvents issues of stability and guarantees that estimates do not explode. The recursion then becomes, 
\begin{align} \label{eq:rec_prel}
& \theta_{n+1}(\lambda) \! =\! 
\theta_{n}(\lambda)\! +\! \frac{1}{n+1}\left(\!1- \frac{\sum_{k=1}^{N_I} p_{\Lambda | X, \Theta}(\lambda | X_{nN_I+k}, \theta_n)}{N_I\rho(\lambda)} \!+\! \beta_n(\lambda) \! +\! Z_n(\lambda)\!\right), 
\end{align}
where $Z_n(\lambda)$ denotes the projection term due to the constraint to $H$. The exact definition of the projection term can be found in \cite[Section 4.3]{KY03}. To sample $X_{nN_I+k}$, $k=1, \dots, N_I$, a Gibbs sampler is used by alternating between sampling $X_{nN_I+k}$ from $q(X_{nN_I+k-1}, x |  \Lambda_{n N_I + k-1})$ and $\Lambda_{nN_I+k}$ from $p_{\Lambda  | X,\Theta}(\lambda | X_{nN_I+k}, \theta_n)$, where $q(x,y |  \lambda)$ has the desired invariant distribution, that is, 
\begin{align*}
    p_{X | \Lambda}(y | \lambda) = \sum_{x \in \mathcal{X}} q(x,y |  \lambda) p_{X | \Lambda}(x | \lambda), \quad y \in \mathcal{X}, \lambda \in \mathcal{L}. 
\end{align*} 
For $n \geq 1$, let $\theta_n = (\theta_n(\lambda):\lambda\in\mathcal{L})$, $\varepsilon_n=1/(n+1)$
\[
    \xi_n = (X_{(n-1)N_I + 1},\dots, X_{n N_I},\Lambda_{(n-1)N_I + 1}, \dots, \Lambda_{n N_I}),
\] and $h: \mathbb{R}^{|\mathcal{L}|} \times (\mathcal{X}^{N_I} \times\mathcal{L}^{N_I}) \to \mathbb{R}^{|\mathcal{L}|}$ be given by
\begin{align*}
    h(\theta,\xi)(\lambda) = 1-\frac{\sum_{k=1}^{N_I} p_{\Lambda | X, \Theta}(\lambda | \xi({k}), \theta)}{N_I\rho(\lambda)},
\end{align*}
where $\xi(k)$ denotes the $k$th component of the vector $\xi$. Notice that $\xi_n(k) = X_{(n-1)N_I+k}$ and $\xi_n(N_I+k) = \Lambda_{(n-1)N_I+k}$ for $k=1,\dots,N_I$.
The recursion \eqref{eq:rec_prel} can then be stated as
\begin{align*}
   \theta_{n+1} &= \theta_{n} + \varepsilon_n\left( h(\theta_n, \xi_{n+1}) + \beta_n + Z_n\right)\\
   &= \theta_{n} + \varepsilon_n Y_n + \varepsilon_n Z_n,
\end{align*}
where $Y_n = h(\theta_n, \xi_{n+1}) + \beta_n$. Let $\mathcal{F}_n=\sigma(\theta_0,Y_0,\dots,Y_{n-1},\xi_1,\dots,\xi_n)$. Then, $P(\xi_{n+1}\in\cdot|\xi_i,\theta_i,i\leq n) = p(\xi_n,\cdot|\theta_n)$, where $p(\xi,\cdot|\theta)$ denotes the one-step transition probability with starting point $\xi$, parameterized by $\theta$, given by
\begin{align} \label{eq:transitionprob}
    p(\xi, \eta | \theta) &= 
    \prod_{k=1}^{N_I} q\left(\eta(k-1),\eta(k)|  \hat\eta(N_I+k -1)\right) p_{\Lambda | X, \Theta}\left(\eta(N_I+k) | \eta(k), \theta\right),    
\end{align}
where $\eta(0) = \xi(N_I)$, $\hat\eta(N_I) = \xi(2N_I)$ and $\hat\eta(k) = \eta(k)$ for all $k\neq N_I$. Furthermore, 
\begin{align*}
    E_n [h(\theta_n,\xi_{n+1})] 
    = E[h(\theta_n,\xi_{n+1})|\mathcal{F}_n] 
    = \sum_{\eta} h(\theta_n,\eta)p(\xi_n,\eta|\theta_n) = g(\theta_n,\xi_n),
\end{align*}
where $E_n$ denotes the expectation conditioned on $\mathcal{F}_n$ and 
\[
    g(\theta,\xi)\doteq \sum_{\eta} h(\theta,\eta)p(\xi,\eta|\theta).
\]
Since $\rho(\lambda)>0$ for all $\lambda\in\mathcal{L}$, we find that $h$ is bounded which implies $E|Y_n|<\infty$ and
\[
    E_n Y_n = g(\theta_n,\xi_n) + \beta_n.
\]
For convenience and because $E_n\beta_n=o(1)$, here we use $\beta_n$ to denote $E_n\beta_n$.

By introducing 
\begin{align*}
    \delta M_n 
    &= Y_n -E_n Y_n = h(\theta_n,\xi_{n+1}) - g(\theta_n,\xi_n)\\
    &= h(\theta_n,\xi_{n+1}) - \sum_{\eta} h(\theta_n,\eta)p(\xi_n,\eta|\theta_n)\\
    &= \sum_{\eta} [h(\theta_n,\xi_{n+1})-h(\theta_n,\eta)]p(\xi_n,\eta|\theta_n),
\end{align*}
the recursion \eqref{eq:rec_prel} may be restated as, 
\[
    \theta_{n+1} = \theta_n + \varepsilon_n [E_n Y_n+\delta M_n+Z_n], 
\] or equivalently,
\begin{align}
\label{eq:recursion}
    \theta_{n+1} = \theta_n + \varepsilon_n [g(\theta_n,\xi_n)+\delta M_n+\beta_n+Z_n],
\end{align}
which is a standard form of a stochastic approximation algorithm with state-dependent noise, see \cite[Section 6.6]{KY03}. 

\subsection{Identification of the limit ODE}
The limit ODE is a projected ODE of the form $\dot \theta  = \bar g(\theta) + z$,  where $z$ is a projection term, that approximates the asymptotic behaviour of the stochastic recursion \eqref{eq:recursion}. To identify $\bar g$ a first step is to assume that the noise $\delta M_n$ and bias $\beta_n$ vanishes asymptotically. Moreover, since $g(\theta,\xi) =  \sum_{\eta} h(\theta,\eta)p(\xi,\eta|\theta)$ and $\theta_n$ varies slowly with $n$, the intuition is that, for large $n$ and $\theta_n \approx \theta$, the Markov chain $\xi_n$ is approximately distributed according to the invariant distribution of the transition probability $p(\xi, \cdot\, | \theta)$. Therefore, a candidate for $\bar g$ is  $E[h(\theta,\xi_1)(\lambda) | \theta_0=\theta]$, where $\xi_1 = (X_{1}, \dots, X_{N_I}, \Lambda_1, \dots, \Lambda_{N_I})$ depends on $\theta_0$ and the expectation is taken under $X_0\overset{d}{=}p_{X | \Theta}(\cdot | \theta)$.

Let $\bar g: \mathbb{R}^{|\mathcal{L}|} \to \mathbb{R}^{|\mathcal{L}|}$ be given by
\begin{align} \label{eq:gbar}
\bar g(\theta)(\lambda) &= 1-\frac{p_{\Lambda | \Theta}(\lambda | \theta)}{\rho(\lambda)}. 
\end{align}
The expression for $\bar g$ is motivated by the following lemma that expresses $\bar g(\theta)$ as the ergodic average of $h(\theta,\xi_n)$, for fixed $\theta$. 
\begin{lemma} \label{lem:gbar}
Let $\bar g$ be given by \eqref{eq:gbar}. For $\theta \in \mathbb{R}^{|\mathcal{L}]}$ and $\lambda\in  \mathcal{L}$, 
$\bar g(\theta)(\lambda) =  E[h(\theta,\xi_1)(\lambda) | \theta_0=\theta]$. 
\end{lemma}
\begin{proof}
For each $\lambda \in \mathcal{L}$, the expectation on the right hand side can be written as
\begin{align*}
    & E[h(\theta,\xi_1)(\lambda) | \theta_0=\theta]   \\
    &= \sum_{x_0 \in \mathcal{X}} E[h(\theta, (X_1, \dots, X_{N_I}, \Lambda_1, \dots, \Lambda_{N_I}))(\lambda) | \theta_0=\theta, X_0 = x_0]  p_{X | \Theta}(x_0 | \theta) \\
    &= \sum_{x_0 \in \mathcal{X}}\sum_{\lambda_0\in \mathcal{L}} E[h(\theta, (X_1, \dots, X_{N_I}, \Lambda_1, \dots, \Lambda_{N_I}))(\lambda) | \theta_0=\theta, X_0 = x_0]  \\ & \qquad \times  p_{X | \Lambda}(x_0 |  \lambda_0)p_{\Lambda | \Theta}(\lambda_0 | \theta)
    \\ &=  \sum_{x_0 \in \mathcal{X}}\cdots\sum_{x_{N_I} \in \mathcal{X}} 
    \sum_{\lambda_0 \in \mathcal{L}}\cdots\sum_{\lambda_{N_I} \in \mathcal{L}}
    h(\theta, (x_1, \dots, x_{N_I}, \lambda_1, \dots, \lambda_{N_I}))(\lambda) \\ & \qquad \times \left(\prod_{k=1}^{N_I}q(x_{k-1},x_{k} |  \lambda_{k-1})p_{\Lambda | X, \Theta}(\lambda_k | x_k, \theta) \right)
    p_{X | \Lambda}(x_0 |  \lambda_0)p_{\Lambda | \Theta}(\lambda_0 | \theta) \\
    &= 
    \sum_{x_0 \in \mathcal{X}} \cdots\sum_{x_{N_I} \in \mathcal{X}}
    \sum_{\lambda_0 \in \mathcal{L}}\cdots\sum_{\lambda_{N_I} \in \mathcal{L}}  \left(1-
    \frac{\sum_{k=1}^{N_I}  p_{\Lambda | X, \Theta}(\lambda | x_{k}, \theta)}{N_I\rho(\lambda)}\right)  
    \\ & \qquad \times \left(\prod_{k=1}^{N_I}q(x_{k-1},x_{k} |  \lambda_{k-1})p_{\Lambda | X, \Theta}(\lambda_k | x_k, \theta) \right)
    p_{X | \Lambda}(x_0 |  \lambda_0)p_{\Lambda | \Theta}(\lambda_0 | \theta)
    \end{align*}
We denote the last term in the previous display by $S_{N_I}$. It remains to show that for any $n\in\mathbb{N}$, 
\[
    S_{n}= 1-\frac{p_{\Lambda | \Theta}(\lambda | \theta)}{\rho(\lambda)}.
\]
For $n=1$, 
\begin{align*}
    S_1 &= \sum_{x_0 \in \mathcal{X}} \sum_{x_1 \in \mathcal{X}}
    \sum_{\lambda_0 \in \mathcal{L}}\sum_{\lambda_1 \in \mathcal{L}}  \left(1-
    \frac{  p_{\Lambda | X, \Theta}(\lambda | x_1, \theta)}{\rho(\lambda)}\right)  
     q(x_{0},x_{1} |  \lambda_{0})
     \\& \qquad\times p_{\Lambda | X, \Theta}(\lambda_1 | x_1, \theta) 
    p_{X | \Lambda}(x_0 |  \lambda_0)p_{\Lambda | \Theta}(\lambda_0 | \theta)\\
    & =  \sum_{x_1 \in \mathcal{X}}
    \sum_{\lambda_0 \in \mathcal{L}} \left(1-
    \frac{  p_{\Lambda | X, \Theta}(\lambda | x_1, \theta)}{\rho(\lambda)}\right) 
    p_{X | \Lambda}(x_1 |  \lambda_0)p_{\Lambda | \Theta}(\lambda_0 | \theta)\\
    & =  \sum_{x_1 \in \mathcal{X}}
     \left(1-\frac{  p_{\Lambda | X, \Theta}(\lambda | x_1, \theta)}{\rho(\lambda)}\right) p_{X | \Theta}(x_1 |  \theta)\\
     & = 1-\frac{p_{\Lambda | \Theta}(\lambda | \theta)}{\rho(\lambda)},
\end{align*}
where the second equation holds since we execute the summation over $x_0$ first with \eqref{eq:invariance} and then we take summation over $\lambda_1$.

Next, if we assume that 
\[
     S_{n}= 1-\frac{p_{\Lambda | \Theta}(\lambda | \theta)}{\rho(\lambda)}
\]
is true, then 
\begin{align*}
    S_{n+1} &= \sum_{x_0 \in \mathcal{X}} \cdots\sum_{x_{n+1} \in \mathcal{X}}
    \sum_{\lambda_0 \in \mathcal{L}}\cdots\sum_{\lambda_{n+1} \in \mathcal{L}}  \left(
    \frac{\sum_{k=1}^{n+1}  [\rho(\lambda)-p_{\Lambda | X, \Theta}(\lambda | x_{k}, \theta)]}{(n+1)\rho(\lambda)}\right)  
    \\ & \qquad \times \left(\prod_{k=1}^{n+1}q(x_{k-1},x_{k} |  \lambda_{k-1})p_{\Lambda | X, \Theta}(\lambda_k | x_k, \theta) \right)
    p_{X | \Lambda}(x_0 |  \lambda_0)p_{\Lambda | \Theta}(\lambda_0 | \theta)\\
    & = \bar S_{n+1} + \hat S_{n+1},
\end{align*}
where 
\begin{align*}
    \bar S_{n+1} &\doteq \sum_{x_0 \in \mathcal{X}} \cdots\sum_{x_{n+1} \in \mathcal{X}}
    \sum_{\lambda_0 \in \mathcal{L}}\cdots\sum_{\lambda_{n+1} \in \mathcal{L}}  \left(
    \frac{  \rho(\lambda)-p_{\Lambda | X, \Theta}(\lambda | x_1, \theta)}{(n+1)\rho(\lambda)}\right)  
    \\ & \qquad \times\left( \prod_{k=1}^{n+1}q(x_{k-1},x_{k} |  \lambda_{k-1})p_{\Lambda | X, \Theta}(\lambda_k | x_k, \theta) \right)
    p_{X | \Lambda}(x_0 |  \lambda_0)p_{\Lambda | \Theta}(\lambda_0 | \theta)
\end{align*}
and 
\begin{align*}
    \hat S_{n+1} &\doteq \sum_{x_0 \in \mathcal{X}} \cdots\sum_{x_{n+1} \in \mathcal{X}}
    \sum_{\lambda_0 \in \mathcal{L}}\cdots\sum_{\lambda_{n+1} \in \mathcal{L}}  \left(
    \frac{\sum_{k=2}^{n+1}  [\rho(\lambda)-p_{\Lambda | X, \Theta}(\lambda | x_{k}, \theta)]}{(n+1)\rho(\lambda)}\right)  
    \\ & \qquad \times \left(\prod_{k=1}^{n+1}q(x_{k-1},x_{k} |  \lambda_{k-1})p_{\Lambda | X, \Theta}(\lambda_k | x_k, \theta) \right)
    p_{X | \Lambda}(x_0 |  \lambda_0)p_{\Lambda | \Theta}(\lambda_0 | \theta).
\end{align*}
For $\bar S_{n+1}$, we first observe that
\begin{align*}
    \sum_{x_2 \in \mathcal{X}} \cdots\sum_{x_{n+1} \in \mathcal{X}}
    \sum_{\lambda_2 \in \mathcal{L}}\cdots\sum_{\lambda_{n+1} \in \mathcal{L}}  \prod_{k=2}^{n+1}q(x_{k-1},x_{k} |  \lambda_{k-1})p_{\Lambda | X, \Theta}(\lambda_k | x_k, \theta) =1.
\end{align*}
It is not hard to check the equality holds by taking summation in the following order: $\lambda_{n+1}\to x_{n+1}\to \lambda_{n}\to x_{n}\to\cdots \to\lambda_{2}\to x_{2}$. This implies that 
\begin{align*}
    \bar S_{n+1} &= \sum_{x_{0} \in \mathcal{X}}\sum_{x_{1} \in \mathcal{X}}\sum_{\lambda_{0} \in \mathcal{L}}
    \sum_{\lambda_{1} \in \mathcal{L}}  \left(
    \frac{  \rho(\lambda)-p_{\Lambda | X, \Theta}(\lambda | x_{n+1}, \theta)}{(n+1)\rho(\lambda)}\right)  
   q(x_{0},x_{1} |  \lambda_{0})\\
   &\qquad \times p_{\Lambda | X, \Theta}(\lambda_{1} | x_{1}, \theta)p_{X | \Lambda}(x_0 |  \lambda_0)p_{\Lambda | \Theta}(\lambda_0 | \theta) \\
   & = \frac{1}{n+1}S_1.
\end{align*}
As for $\hat S_{n+1}$, since 
\begin{align*}
    &\sum_{x_0 \in \mathcal{X}}\sum_{\lambda_0 \in \mathcal{L}} q(x_{0},x_{1} |  \lambda_{0}) p_{\Lambda | X, \Theta}(\lambda_{1} | x_{1}, \theta)p_{X | \Lambda}(x_0 |  \lambda_0)p_{\Lambda | \Theta}(\lambda_0 | \theta)\\
    & = \sum_{\lambda_0 \in \mathcal{L}}  p_{\Lambda | X, \Theta}(\lambda_{1} | x_{1}, \theta)p_{X | \Lambda}(x_1 |  \lambda_0)p_{\Lambda | \Theta}(\lambda_0 | \theta)\\
    & = p_{\Lambda | X, \Theta}(\lambda_{1} | x_{1}, \theta)p_{X | \Theta}(x_1 |  \theta)\\
    & =  p_{X | \Lambda}(x_1 |  \lambda_1)p_{\Lambda | \Theta}(\lambda_1 | \theta),
\end{align*}
we have 
\begin{align*}
    \hat S_{n+1} &= \sum_{x_1 \in \mathcal{X}} \cdots\sum_{x_{n+1} \in \mathcal{X}}
    \sum_{\lambda_1 \in \mathcal{L}}\cdots\sum_{\lambda_{n+1} \in \mathcal{L}}  \left(
    \frac{\sum_{k=2}^{n+1}  [\rho(\lambda)-p_{\Lambda | X, \Theta}(\lambda | x_{k}, \theta)]}{(n+1)\rho(\lambda)}\right)  
    \\ & \qquad \times \left(\prod_{k=2}^{n+1}q(x_{k-1},x_{k} |  \lambda_{k-1})p_{\Lambda | X, \Theta}(\lambda_k | x_k, \theta) \right)
    p_{X | \Lambda}(x_1 |  \lambda_1)p_{\Lambda | \Theta}(\lambda_1 | \theta)\\
    & = \frac{n}{n+1}S_n.
\end{align*}
Hence, 
\[
    S_{n+1} = \bar S_{n+1}+\hat S_{n+1} = \frac{1}{n+1}S_1+\frac{n}{n+1}S_n = 1-\frac{p_{\Lambda | \Theta}(\lambda | \theta)}{\rho(\lambda)}.
\]
We complete the proof by mathematical induction.
\end{proof}

The limit ODE can now be identified by
\begin{align}\label{eq:limitode}
    \dot \theta  = \bar g(\theta) + z, 
\end{align}
where $z \in -C(\theta)$ is the reflection due to the projection onto the hyperrectangle $H$. For $\theta$ in the interior of $H$, $C(\theta) = \{0\}$; for $\theta$ on the boundary of $H$, $C(\theta)$ is the infinite convex cone generated by the outer normals at $\theta$ of the faces on which $\theta$ lies, see \cite[Section 4.3]{KY03} for additional details. 

To identify a set containing the limit points of the stochastic approximation algorithm, it is sufficient to construct a Lyapunov function, $V$, for the limit ODE and find the level set of $0$.
If $\left\langle \nabla V(\theta),\bar g(\theta)\right\rangle \leq 0$, with equality if and only if $V(\theta) = 0$, then the Lyapunov function decreases under the dynamics of the limit ODE until it  reaches $V(\theta(t)) = 0$, after which it remains constant at $0$. Indeed,
\[
 V(\theta(t_1)) - V(\theta(t_0)) = \int_{t_0}^{t_1} \langle \nabla V(\theta(t)), \dot \theta(t) \rangle dt = \int_{t_0}^{t_1} \langle \nabla V(\theta(t)), \bar g(\theta(t)) \rangle dt  < 0, 
\]
where the reflection term has been ignored for convenience. Consequently, the limit set of the ODE is contained in the level set where $V(\theta) = 0$. 

 To construct a Lyapunov function for the limit ODE of the AWH algorithm,  let $V:\mathbb{R}^{|\mathcal{L}|} \to [0,\infty)$ be given by
\begin{align}\label{eq:V}
    V(\theta) = \sum_{\sigma \in \mathcal{L}}\rho(\sigma) [\bar g(\theta)(\sigma)]^2 = 
    \sum_{\sigma \in \mathcal{L}}\rho(\sigma) \left(1-\frac{p_{\Lambda | \Theta}(\sigma | \theta)}{\rho(\sigma)}\right)^2. 
\end{align}

\begin{lemma} 
\label{lem_4}
Let $V$ be given by \eqref{eq:V}. Then $V$ is continuously differentiable and $\left\langle \nabla V(\theta),\bar g(\theta)\right\rangle \leq 0$ with equality if and only if $V(\theta) = 0$.
\end{lemma}

\begin{proof} 
It follows immediately from the definition that $V$ is continuously differentiable. It remains
to prove $\left\langle \nabla V(\theta),\bar g(\theta) \right\rangle \leq 0$ with
equality if and only if $V(\theta) = 0$. To determine the gradient, $\nabla V$, note first that since $p_{\Lambda | \Theta}$ is given by \eqref{eq:plambda}, 
\begin{align*}
    \frac{\partial p_{\Lambda | \Theta}(\sigma | \theta)}{\partial \theta(\lambda)} = \left\{\begin{array}{ll} 
    - p_{\Lambda | \Theta}(\sigma | \theta)p_{\Lambda | \Theta}(\lambda | \theta), & \sigma \neq \lambda, \\
    p_{\Lambda | \Theta}(\lambda | \theta)(1-p_{\Lambda | \Theta}(\lambda | \theta)), & \sigma = \lambda.\end{array}\right.
\end{align*}
Consequently, the gradient of $V$ is given by
\begin{align*}
\frac{\partial V(\theta)}{\partial \theta(\lambda)} & =\frac{\partial}{\partial \theta(\lambda)}\left(\sum_{\sigma \in \mathcal{L}}\rho(\sigma)\left(1-\frac{p_{\Lambda | \Theta}(\sigma | \theta)}{\rho(\sigma)}\right)^{2}\right)\\
 & =
 \sum_{\sigma \in \mathcal{L}}2
 \left(1-\frac{p_{\Lambda | \Theta}(\sigma | \theta)}{\rho(\sigma)}\right) \left(-\frac{\partial p_{\Lambda | \Theta}(\sigma | \theta)}{\partial \theta(\lambda)}\right)\\
 & =2 \sum_{\sigma \in \mathcal{L}} \left(1-\frac{p_{\Lambda | \Theta}(\sigma | \theta)}{\rho(\sigma)}\right) p_{\Lambda | \Theta}(\sigma | \theta)p_{\Lambda | \Theta}(\lambda | \theta)
 \\ & \quad - 2\left(1-\frac{p_{\Lambda | \Theta}(\lambda | \theta)}{\rho(\lambda)}\right)p_{\Lambda | \Theta}(\lambda | \theta) \\
  & =2 \sum_{\sigma \in \mathcal{L}} \bar g(\theta)(\sigma) p_{\Lambda | \Theta}(\sigma | \theta)p_{\Lambda | \Theta}(\lambda | \theta)- 2\bar g(\theta)(\lambda)p_{\Lambda | \Theta}(\lambda | \theta).
\end{align*}
Taking the inner product with $\bar g(\theta)$ yields,
\begin{align*}
 \frac{1}{2}\left\langle \nabla V(\theta),\bar g(\theta)\right\rangle
 & =\sum_{\lambda \in \mathcal{L}}\left[\sum_{\sigma \in \mathcal{L}} \bar g(\theta)(\sigma) p_{\Lambda | \Theta}(\sigma | \theta)
 p_{\Lambda | \Theta}(\lambda | \theta)\right] \bar g(\theta)(\lambda)\\
 & \qquad-\sum_{\lambda \in \mathcal{L}} \bar g(\theta)(\lambda)p_{\Lambda | \Theta}(\lambda | \theta) \bar g(\theta)(\lambda)\\
  & =\left(\sum_{\lambda \in \mathcal{L}} \bar g(\theta)(\lambda)p_{\Lambda | \Theta}(\lambda | \theta)\right)^{2}
 - \sum_{\lambda \in \mathcal{L}}  [\bar g(\theta)(\lambda)]^2 p_{\Lambda | \Theta}(\lambda | \theta) \\
 &= - \text{Var}_{p_{\Lambda | \Theta}(\cdot | \theta)}(\bar g(\theta)) \\
 & \leq 0.
\end{align*}
In the last equality it is observed that the desired expression can be represented as the variance of $\bar g(\theta)$ under the distribution $p_{\Lambda | \Theta}(\cdot | \theta)$ on $\mathcal{L}$. Consequently, $\left\langle \nabla V(\theta),\bar g(\theta)\right\rangle =0$
if and only if $\bar g(\theta)$ is a constant, i.e., there exists a constant $c\in\mathbb{R}$ such that
\[
 \bar g(\theta)(\lambda)=1-\frac{p_{ \Lambda|\Theta}(\lambda | \theta)}{\rho(\lambda)} =c, \quad \text{ for all }\lambda \in\mathcal{L}.
\]
Multiplying both sides by $\rho(\lambda)$ and summing over $\lambda$, leads to
\begin{align*}
c & =\sum_{\lambda \in \mathcal{L}} c \rho(\lambda)=\sum_{\lambda \in \mathcal{L}} \left(\rho(\lambda)-p_{\Lambda | \Theta}(\lambda | \theta)\right)
=0.
\end{align*}
This implies that $c$ must be equal to $0$, and thus $\left\langle \nabla V(\theta),\bar g(\theta)\right\rangle =0$
if and only if
\[
V(\theta)=\sum_{\lambda \in \mathcal{L}} \rho(\lambda) [\bar g(\theta)(\lambda)]^{2} = c^2 = 0. 
\]
This completes the proof.
\end{proof}

\subsection{Convergence of estimates of free energy differences}

In this subsection, the main theorem on the almost sure convergence of the estimates of the design parameters $\{\theta_n\}$ of the AWH algorithm is provided in Theorem \ref{thm:asconvergence}. The subsequent Corollary \ref{cor:fediff} provides the result on the almost sure convergence of estimates of free energy differences. 

\begin{theorem}\label{thm:asconvergence}
Let $\{\theta_n\}$ be given by the AWH algorithm \eqref{eq:recursion}, with target distribution $\rho$, $\rho(\lambda) > 0, \lambda \in \mathcal{L}$, let $H$ a hyper-rectangle in $\mathbb{R}^{|\mathcal{L}|}$ and $V$ be given by \eqref{eq:V}. Assume that $H$ is large enough so that $\{\theta\in H:V(\theta)=0\}$ is non-empty. Then  $\{\theta_n\}$ converges almost surely to $\{\theta \in H: V(\theta) = 0\}$. 
\end{theorem}

\begin{proof}
As noted previous, the recursion of the AWH can be written as a stochastic approximation algorithm with state-dependent noise as in \eqref{eq:recursion}.  It follows from Theorem \ref{thm_1} in the Appendix, once the conditions of that theorem are verified, that $\{\theta_n\}$ converges almost surely to some limit set of the limit ODE, given by \eqref{eq:limitode}. By Lemma \ref{lem_4} it follows that the limit set is contained in $\{\theta \in H: V(\theta) = 0\}$, see e.g.\ \cite[Theorem 4.2.3]{KY03}. 

To complete the proof, the following conditions of Theorem \ref{thm_1} must be verified: (5.1.1), (A1.1), (A1.2), (A1.4), (6.2), (A6.1), (A6.2), and (A4.3.1).  

Condition (5.1.1) is satisfied since    $\varepsilon_n = 1/(n+1)$, for $n \geq 0$, and consequently, $\sum_{n=0}\varepsilon_n  = \infty$, $\varepsilon_n \geq 0$, and $ \lim_{n \to \infty} \varepsilon_n = 0$. To verify (A1.1), the upper bound
    \begin{align}
        \label{eq:hbound}
        \left|h(\theta,\xi)(\lambda)\right| = \left|1-\frac{\sum_{k=1}^{N_I} p_{\Lambda | X, \Theta}(\lambda | \xi_{k}, \theta)}{N_I\rho(\lambda)}\right|\leq 1+\frac{1}{\rho(\lambda)}, 
 \end{align}
    implies that  $|h(\theta,\xi)(\lambda)|\leq C$ with $C\doteq 1 + \max_{\lambda\in\mathcal{L}}(1/\rho(\lambda))<\infty$ for all $\theta,\xi$ and $\lambda$. Consequently,  
    $\sup_n E|Y_n|=\sup_n E|h(\theta_n,\xi_{n+1})|\leq C <\infty.$
The condition (A1.2) states that $g_n(\theta, \xi) = g(\theta,\xi)$ is continuous in $\theta$
for each $\xi$. It follows immediately from the definitions, since both $p_{\Lambda | X, \Theta}(\lambda | \xi(k), \theta)$ and $p(\xi,\eta | \theta)$ are continuous in $\theta$ for each $\xi,\eta$, $\lambda$ and $k\in\{1,2,\dots,N_I\}$. To verify condition (A1.4) it is sufficient to prove
\begin{align*}
    \text{(A2.1)} \quad \sup_n E|Y_n|^2<\infty \quad \text{and} \qquad 
    \text{(A2.4)} \quad \sum_n \varepsilon_n^2 <\infty.
\end{align*}
Indeed, according to the last paragraph on p.\ 137 in \cite{KY03}, (A2.1) and (A2.4) implies the stronger statement (3.2) in \cite{KY03}, which says for some $T>0$,
    \[
        \lim_{n\rightarrow\infty}\left(\sup_{j\geq n}\max_{0\leq t\leq T} \left|\sum_{i=m(jT)}^{m(jT+t)-1}\varepsilon_i\delta M_i\right|\right) = 0, \quad  \text{ a.s}.
    \]
    Thus, for each $\mu>0$ and some $T>0$,
    \[
        \lim_{n\rightarrow\infty}P\left\{ \sup_{j\geq n}\max_{0\leq t\leq T}\left|\sum_{i=m(jT)}^{m(jT+t)-1}\varepsilon_i\delta M_i\right|\geq \mu\right\} = 0, 
    \]
    which is condition (A1.4). The conditions (A2.1) and (A2.4) are satisfied by the definition of $\varepsilon_n$ and since the function $h$ is bounded by \eqref{eq:hbound}.
 Condition (6.2) states that  $P(\xi_{n+1}\in\cdot|\xi_i,\theta_i,i\leq n) = p(\xi_n,\cdot|\theta_n)$, where $p(\xi,\cdot|\theta)$ denotes the one-step transition probability with starting point $\xi$ and parameterized by $\theta$, which has already been observed in connection to \eqref{eq:transitionprob}. 
 
 To verify conditions (A6.1) and (A6.2), recall that sufficient conditions for them (provided in \cite[Section 5.3]{KY03}) are \\
(A3.1) For each $\mu>0$,
\[
    \sum_n e^{-\mu/\varepsilon_n}<\infty.
\]
(A3.2) For some $T<\infty$, there is a $c_1(T)<\infty$ such that for all $n$,
\[
    \sup_{n\leq i\leq m(t_n+T)}\frac{\varepsilon_i}{\varepsilon_n}\leq c_1(T).
\]
(A3.3) There is a real $K<\infty$ such that for small real $\gamma$, all $n$, and each component $\alpha_{n,j}$ of $\alpha_n$, $\delta N_{n,j}$ of $\delta N_{n}$,
\[
    E_n e^{\gamma \alpha_{n,j}}\leq e^{\gamma^2K/2}\quad\mbox{ and }\quad E_n e^{\gamma (\delta N_{n,j})}\leq e^{\gamma^2K/2}.
\]
One can easily verify that conditions (A3.1) and (A3.2) hold when $\varepsilon_n =1/(n+1)$. It remains to check if (A3.3) holds for both  $\alpha_n$ and $\delta N_n$. To prove this, let $\bar{g}$ be given by \eqref{eq:gbar} and note that, by Lemma \ref{lem:gbar} 
    \[
        \bar g(\theta) = E[g(\theta,\xi_1) | \theta_0=\theta], 
    \]
    where the expectation is taken under the invariant distribution of $\xi_1$ given $\theta_0 = \theta$. Since $p_{\Lambda | \Theta}(\lambda | \theta)$ is continuous in $\theta$ for each $\lambda \in \mathcal{L}$ it follows that $\bar{g}$ is continuous.  Moreover, since the state space $\mathcal{X}^{N_I} \times \mathcal{L}^{N_I}$ of $\{\xi_n\}$ is finite, and $\{\xi_n\}$ is irreducible and aperiodic,  
    $\{\xi_n\}$ is a geometrically ergodic Markov chain. Indeed, the irreducibility and aperiodicity of $\{\xi_n\}$ follows from the definition of its transition probability $p(\xi, \eta | \theta)$ in \eqref{eq:transitionprob}, the assumption of irreducibility and aperiodicity on $\{q(x,y|\lambda)\}_{x,y}$ and the definition of $p_{\Lambda | X, \Theta}(\lambda | x, \theta) $. Consequently, the summands in the definition of $v_n(\theta,\xi)$ converges to zero at a geometric rate. This implies that $\alpha_n$ and $\delta N_n$ are bounded. As a result, they both satisfy (A3.3) and the asymptotic rates of change of the processes $A^0(t)$ and $N^0(t)$ are zero with probability one. 
    Lastly, (A4.3.1) holds due to our assumption on $H$. This completes the proof. 
\end{proof}


\begin{corollary} \label{cor:fediff}
Let $\{\theta_n\}$ be given by the AWH algorithm \eqref{eq:recursion}, with target distribution $\rho$, $\rho(\lambda) > 0, \lambda \in \mathcal{L}$ and $H$ a hyper-rectangle in $\mathbb{R}^{|\mathcal{L}|}$ such that $\{\theta \in H: V(\theta) = 0\}$ is non-empty, where $V$ is given by \eqref{eq:V}. Then, for each $\lambda,\sigma \in \mathcal{L}$, the sequence $\{\theta_n(\lambda) - \theta_n(\sigma)\}$,  converges almost surely to $F(\lambda)-F(\sigma) + \log \rho(\lambda)-\log \rho(\sigma)$.
\end{corollary}

 \begin{proof}
 By Theorem \ref{thm:asconvergence} the sequence $\{\theta_n\}$ converges almost surely to the set $\{\theta \in H : V(\theta) = 0\}$. For any $\theta \in \mathbb{R}^{|\mathcal{L}|}$ with $V(\theta) = 0$ it follows that $p_{\Lambda | \Theta}(\lambda | \theta) = \rho(\lambda)$, for each $\lambda \in \mathcal{L}$, and consequently,
 \begin{align*}
     \theta(\lambda) - F(\lambda) = \log \rho(\lambda) + \log \sum_{\tau \in \mathcal{L}} e^{\theta(\tau)-F(\tau)}. 
 \end{align*}
 Hence, for $\lambda, \sigma \in \mathcal{L}$, 
  \begin{align*}
     \theta(\lambda) - \theta(\sigma) =  F(\lambda)-F(\sigma) + \log \rho(\lambda) - \log \rho(\sigma). 
 \end{align*}
 This completes the proof. 
 \end{proof}

From Corollary \ref{cor:fediff} it follows that estimates of free energy differences can be obtained by running the AWH algorithm for $N$ iterations, to obtain $\theta_N$, and then estimate the free energy difference $F(\lambda) - F(\sigma)$ by
\begin{align*}
    \theta_N({\lambda)}-\theta_N(\sigma)- \log \rho(\lambda) + \log \rho(\sigma). 
\end{align*}

\section{Almost Sure Convergence of Ergodic Averages} \label{sec:erg}

In addition to estimating free energy differences, the AWH algorithm may be used to estimate expectations at a fixed parameter value. Namely, for a fixed $\lambda \in \mathcal{L}$ and a function $\psi:\mathcal{X}\rightarrow\mathbb{R}$, consider estimating 
\begin{align}\label{eq:ergavg}
    \sum_{x\in\mathcal{X}}\psi(x)p_{X | \Lambda}(x | \lambda),
\end{align}
which assumed to be finite.

At the $n$th iteration of the AWH algorithm, given the free energy estimate $\theta_{n}$, the Markov chain $\{(X_{nN_I + k},\Lambda_{nN_I + k}), k=1, \dots, N_I\}$ has invariant distribution $p_{X,\Lambda | \Theta}(x,\lambda | \theta_n)$ given by \eqref{eq:pjoint}. The estimate of \eqref{eq:ergavg} that we proposed in Subsection \ref{subsec:2.2} is given by
\begin{align} \label{eq:ergest}
    \frac{\sum_{n=0}^{N-1} \sum_{k=1}^{N_I}\psi(X_{nN_I + k})p_{\Lambda | X, \Theta}(\lambda | X_{nN_I + k}, \theta_n)}{\sum_{n=0}^{N-1}\sum_{k=1}^{N_I}p_{X | \Lambda, \Theta}(\lambda | X_{nN_I + k}, \theta_n)}.  
\end{align}
To study the almost sure convergence of the estimator in the last display, it is sufficient to study the almost sure convergence of
\begin{align*}
    \hat \phi_{N} = \frac{1}{N\,N_I}\sum_{n=0}^{N-1} \sum_{k=1}^{N_I}\phi(X_{nN_I + k})p_{\Lambda | X, \Theta}(\lambda | X_{nN_I + k}, \theta_n),
\end{align*}
which is an estimate of 
\begin{align*}
   \sum_{x \in \mathcal{X}} \phi(x) p_{ \Lambda | X, \Theta}(\lambda | x, \theta_*)p_{X | \Theta}(x |\theta_* ) =  \sum_{x \in \mathcal{X}} \phi(x) p_{X | \Lambda}(x | \lambda)p_{\Lambda | \Theta}(\lambda | \theta_*),
\end{align*}
where the optimal design parameter $\theta_*$ satisfying
\begin{align*}
  p_{\Lambda | \Theta}(\lambda | \theta_*) = \rho(\lambda), \quad \lambda \in \mathcal{L}. 
\end{align*}
With $\phi = \psi$ and $\phi \equiv 1$ it follows that the estimator \eqref{eq:ergest} is given by $\hat \psi_{N} / \hat 1_{N}$, which estimates the desired quantity
\begin{align*}
    \frac{\sum_{x \in \mathcal{X}} \psi(x) p_{X | \Lambda}(x | \lambda)p_{\Lambda | \Theta}(\lambda | \theta_*)}{\sum_{x \in \mathcal{X}}  p_{X | \Lambda}(x | \lambda) p_{\Lambda | \Theta}( \lambda | \theta_*)} 
    = \sum_{x \in \mathcal{X}} \psi(x) p_{X | \Lambda}(x | \lambda).
\end{align*}
In order to show the almost sure convergence of the sequence $\{\hat \phi_{n}\}_{n \geq 1}$, observe that $\{\hat \phi_{n}\}_{n \geq 1}$ can be written as an iteration, with $\hat \phi_{0} = 0$ and for $n \geq 0$, 
\begin{align*}
    \hat \phi_{n+1} = \hat \phi_{n} + \frac{1}{n+1}\left(\Phi(\xi_{n}, \theta_n) -  \hat \phi_{n}\right),
\end{align*} 
where  
\[
\Phi(\xi, \theta) = \frac{1}{N_I} \sum_{k=1}^{N_I} \phi\left(\xi(k)\right) p_{\Lambda | X, \Theta}\left(\lambda | \xi(k), \theta\right),  \; \theta \in \mathbb{R}^{|\mathcal{L}|}
\] and, as in the previous section, 
\[
    \xi_n = (X_{(n-1)N_I + 1},\dots, X_{n N_I},\Lambda_{(n-1)N_I + 1}, \dots, \Lambda_{n N_I}).
\]
Consider the extended design parameter $\vartheta = (\theta, \zeta) \in \mathbb{R}^{|\mathcal{L}|} \times [-C,C]$ with $C\doteq \max_{x\in\mathcal{X}}|\phi(x)|<\infty$, and without loss of generality, we assume $0\not\in\mathcal{L}$ and define $\mathcal{L}_0=\mathcal{L}\cup\{0\}$. For $u \in \mathcal{L}$ we interpret $\vartheta(u) = \theta(u)$, and for $u =0$,  $\vartheta(u) = \zeta$. The optimal value of extended design parameter is $\vartheta_* = (\theta_*, \zeta_*)$, where 
$\zeta_* = \sum_{x \in \mathcal{X}}\phi(x) p_{X | \Lambda}(x | \lambda)p_{\Lambda | \Theta}(\lambda | \theta_*)\in [-C,C]$. The estimator of the optimal extended design parameter is
\[
\vartheta_n = (\theta_n, \zeta_n)=((\theta_n(\sigma):\sigma \in\mathcal{L}), \hat  \phi_{n}),
\]
$\varepsilon_n=1/(n+1)$ and the extended function $h: (\mathbb{R}^{|\mathcal{L}| }\times [-C,C]) \times (\mathcal{X}^{N_I}\times \mathcal{L}^{N_I}) \to \mathbb{R}^{|\mathcal{L}_0| }$ given by
\begin{align*}
    h(\vartheta,\xi)(u) = \left\{\begin{array}{ll} 
    1-\frac{\sum_{k=1}^{N_I} p_{\Lambda | X, \Theta}(u | \xi({k}), \theta)}{N_I\rho(u)}, & \text{ if } u \in \mathcal{L}, \\
    \Phi(\xi, \theta) - \zeta, & \text{ if }  u =0. \end{array} \right. 
\end{align*}
The recursion of the AWH algorithm for the extended design parameter  can be stated as
\begin{align*}
   \vartheta_{n+1} &= \vartheta_{n} + \varepsilon_n\left( h(\vartheta_n, \xi_{n+1}) + \beta_n + Z_n\right)\\
   &= \vartheta_{n} + \varepsilon_n Y_n + \varepsilon_n Z_n,
\end{align*}
where $Y_n = h(\vartheta_n, \xi_{n+1}) + \beta_n$. Let $\mathcal{F}_n=\sigma(\theta_0,Y_0,\dots,Y_{n-1},\xi_1,\dots,\xi_n)$,  $p(\xi,\cdot | \theta)$ be as in \eqref{eq:transitionprob}, and  
\[
    g(\vartheta,\xi)\doteq \sum_{\eta} h(\vartheta,\eta)p(\xi,\eta|\theta).
\]
Since $h$ is bounded, it follows that $E|Y_n|<\infty$ and
\[
    E_n Y_n = g(\vartheta_n,\xi_n) + \beta_n.
\]
By introducing 
\begin{align*}
    \delta M_n 
    &= Y_n -E_n Y_n = h(\vartheta_n,\xi_{n+1}) - g(\vartheta_n,\xi_n)\\
    &= h(\vartheta_n,\xi_{n+1}) - \sum_{\eta} h(\vartheta_n,\eta)p(\xi_n,\eta|\theta)\\
    &= \sum_{\eta} [h(\vartheta_n,\xi_{n+1})-h(\vartheta_n,\eta)]p(\xi_n,\eta|\theta)
\end{align*}
the recursion may be restated as 
\[
    \vartheta_{n+1} = \vartheta_n + \varepsilon_n [E_n Y_n+\delta M_n+Z_n]
\] or equivalently,
\begin{align}
\label{eq:recursion_ext}
    \vartheta_{n+1} = \vartheta_n + \varepsilon_n [g(\vartheta_n,\xi_n)+\delta M_n+\beta_n+Z_n],
\end{align}
which, again, is a standard form of a stochastic approximation algorithm with state-dependent noise, see \cite[Section 6.6]{KY03}. 

\subsection{Identification of the limit ODE}
To derive the limit ODE we proceed as in the previous section with the extended design parameter $\vartheta = (\theta, \zeta)$. Let $\bar g: \mathbb{R}^{|\mathcal{L}| }\times [-C,C] \to 
\mathbb{R}^{|\mathcal{L}_0|}$ be given by
\begin{align} \label{eq:gbar2}
\bar g(\vartheta)(u) &= \left\{\begin{array}{ll} 
1-\frac{p_{\Lambda | \Theta}(u | \theta)}{\rho(u)}, & \text{if } u \in \mathcal{L}, \\
\sum_{\mathcal{X}} \phi(x) p_{X, \Lambda | \Theta}(x, \lambda | \theta) - \zeta, & \text{if } u =0. 
\end{array} \right.
\end{align}
Similar to Section \ref{sec:convergence} the expression for $\bar g$ is the ergodic average of $h(\vartheta,\xi_1)$, for fixed $\vartheta_0 = (\theta,\zeta)$. That is, the expectation over $\xi_1 = (X_{1}, \dots, X_{N_I}, \Lambda_1, \dots, \Lambda_{N_I})$ where $X_{0}$ is sampled from $p_{X | \Theta}(\cdot |  \theta)$. 

\begin{lemma} \label{lem:gbar2}
Let $\bar g$ be given by \eqref{eq:gbar2}. For $\vartheta \in \mathbb{R}^{|\mathcal{L}| }\times [-C,C]$ and $u \in  \mathcal{L}_0$, 
$\bar g(\vartheta)(u) =  E[h(\vartheta,\xi_1)(u) |  \vartheta_0 = (\theta,\zeta)]$. 
\end{lemma}

\begin{proof}
For $u \in \mathcal{L}$ the statement reduces to that of Lemma \ref{lem:gbar}, so it is sufficient to consider $u=0 $. The expectation on the right hand side can be written
\begin{align*}
    & E[h(\vartheta,\xi_1)(u) | \vartheta_0 = (\theta,\zeta)]   \\
    & = E[h(\vartheta, (X_1, \dots, X_{N_I}, \Lambda_1, \dots, \Lambda_{N_I}))(u) | \vartheta_0 = (\theta,\zeta)] \\
    & = 
    \sum_{x_0 \in \mathcal{X}} \cdots\sum_{x_{N_I} \in \mathcal{X}}
    \sum_{\lambda_0 \in \mathcal{L}}\cdots\sum_{\lambda_{N_I} \in \mathcal{L}}  \left(\frac{1}{N_I}\sum_{k=1}^{N_I} \phi(x_{k})p_{\Lambda | X,\Theta}(\lambda | x_k, \theta) - \zeta \right)  
    \\ & \qquad \times \left(\prod_{k=1}^{N_I}q(x_{k-1},x_{k} |  \lambda_{k-1})p_{\Lambda | X, \Theta}(\lambda_k | x_k, \theta) \right)
    p_{X | \Lambda}(x_0 |  \lambda_0)p_{\Lambda | \Theta}(\lambda_0 | \theta) 
\end{align*}
By applying the mathematical induction argument as in the proof of Lemma \ref{lem:gbar} with minor modifications, we can find that the last term in the previous display equals 
\[
    \sum_{\mathcal{X}} \phi(x) p_{X, \Lambda | \Theta}(x, \lambda | \theta) - \zeta.
\]
This completes the proof. 
\end{proof}

The limit ODE is given by
\begin{align}\label{eq:limitode2}
    \dot \vartheta  = \bar g(\vartheta) + z, 
\end{align}
where $z \in -C(\vartheta)$ is the reflection. To construct a Lyapunov function for the limit ODE, recall that 
$\theta_*$ denotes the optimal design parameter and let, for $\delta > 0$,  $V_\delta:\mathbb{R}^{|\mathcal{L}| }\times [-C,C] \to [0,\infty)$ be given by
\begin{align}\label{eq:Verg}
    V_\delta(\vartheta)\! =\! 
    \sum_{\sigma \in \mathcal{L}}\rho(\sigma) \left(1\!-\! \frac{p_{\Lambda | \Theta}(\sigma | \theta)}{\rho(\sigma)}\right)^2
     \!\! +  \delta \left(\sum_{x\in\mathcal{X}} \phi(x) p_{X,\Lambda | \Theta}(x,\lambda | \theta_*) \!-\! \zeta\right)^2\!\!.  
\end{align}

\begin{lemma} 
\label{lem_4_2}
Let $V_\delta$ be defined as in \eqref{eq:Verg} with $\delta = m^2/(2C^2)$, where  $C\doteq \max_{x\in\mathcal{X}}|\phi(x)|$ and
\[
    m\doteq \inf_{\theta\in H,\sigma\in\mathcal{L}}p_{\Lambda | \Theta}(\sigma| \theta)=\inf_{\theta\in H,\sigma\in\mathcal{L}}\left(\frac{e^{\theta(\sigma)-F(\sigma)}}{\sum_{\tau\in\mathcal{L}}e^{\theta(\tau)-F(\tau)}}\right)>0.
\]
Then $V_\delta$ is continuously differentiable and, for $\delta > 0$ sufficiently small, $\left\langle \nabla V_\delta(\vartheta),\bar g(\vartheta)\right\rangle \leq 0$ with equality if and only if $V_\delta(\vartheta) = 0$.
\end{lemma}

\begin{proof} 
It follows immediately from the definition that $V_\delta$ is continuously differentiable. It remains
to prove $\left\langle \nabla V_\delta(\vartheta),\bar g(\vartheta) \right\rangle \leq 0$ with
equality if and only if $V(\vartheta) = 0$.

From the proof of Lemma \ref{lem_4} it follows that for $u \in \mathcal{L}$
\begin{align*}
\frac{\partial V_\delta(\vartheta)}{\partial \vartheta(u)}  &=
 \sum_{\sigma \in \mathcal{L}}2
 \left(1-\frac{p_{ \Lambda|\Theta}(\sigma | \theta)}{\rho(\sigma)}\right) \left(-\frac{\partial p_{\Lambda | \Theta}(\sigma | \theta)}{\partial \theta(u)}\right) \\
  & =2 \sum_{\sigma \in \mathcal{L}} \bar g(\vartheta)(\sigma) p_{\Lambda | \Theta}(\sigma | \theta)p_{\Lambda | \Theta}(u | \theta)- 2\bar g(\vartheta)(u)p_{\Lambda | \Theta}(u | \theta). 
\end{align*}
For $u =0$ we have
\begin{align*}
\frac{\partial V_\delta(\vartheta)}{\partial \vartheta(u)} & =
- 2 \delta \bar g((\theta_*,\zeta))(u). 
\end{align*}
Taking the inner product with $\bar g(\vartheta)$ yields,
\begin{align*}
 &\frac{1}{2}\left\langle \nabla V_\delta(\vartheta),\bar g(\vartheta)\right\rangle\\
 & =\sum_{\tau \in \mathcal{L}}\left[\sum_{\sigma \in \mathcal{L}} \bar g(\vartheta)(\sigma) p_{\Lambda | \Theta}(\sigma | \vartheta)
 p_{\Lambda | \Theta}(\tau | \vartheta)\right] \bar g(\vartheta)(\tau)\\
 & \qquad-\sum_{\tau \in \mathcal{L}} \bar g(\vartheta)(\tau)p_{\Lambda | \Theta}(\tau | \vartheta) \bar g(\vartheta)(\tau) - \delta \bar g((\theta_*,\zeta))(0) \cdot  \bar g((\theta,\zeta))(0) \\
  & =\left(\sum_{\tau \in \mathcal{L}} \bar g(\vartheta)(\tau)p_{\Lambda | \Theta}(\tau | \vartheta)\right)^{2}
 - \sum_{\tau \in \mathcal{L}}  [\bar g(\vartheta)(\tau)]^2 p_{\Lambda | \Theta}(\tau | \vartheta) \\
 & \qquad-  \delta \bar g((\theta_*,\zeta))(0) \cdot  \bar g((\theta,\zeta))(0)  \\
  & \leq\left(\sum_{\tau \in \mathcal{L}} \bar g(\vartheta)(\tau)p_{\Lambda | \Theta}(\tau | \vartheta)\right)^{2}
 - \sum_{\tau \in \mathcal{L}}  [\bar g(\vartheta)(\tau)]^2 p_{\Lambda | \Theta}(\tau | \vartheta) \\
&\qquad + \frac{\delta}{2} \left[\bar g((\theta_*,\zeta))(0) -\bar g((\theta,\zeta))(0)\right] \left[\bar g((\theta_*,\zeta))(0)-\bar g((\theta,\zeta))(0)\right] \\ 
& \qquad - \frac{\delta}{2} [\bar g((\theta,\zeta))(0)]^2. 
\end{align*}
The inequality holds since for any $x,y\in\mathbb{R}$, $-2xy\leq (x-y)^2-y^2$.
Note that the first and third terms on the right of the inequality are always non-positive. The middle term can be written as
\begin{align}\label{eq:middle}
    \frac{\delta}{2}  \left( \sum_{x\in\mathcal{X}} \phi(x)p_{X| \Lambda}(x | \lambda)\right)^2\left(p_{\Lambda | \Theta}(\lambda | \theta_*) - p_{\Lambda | \Theta}(\lambda | \theta) \right)^2. 
\end{align}
The sum appearing in the square is bounded by $C\doteq \max_{x\in\mathcal{X}} |\phi(x)|<\infty$, and since 
$p_{\Lambda | \Theta}(\lambda | \theta_*) = \rho(\lambda)$ and $\rho(\lambda) \leq 1$, the expression in \eqref{eq:middle} can be bounded above by 
\begin{align*}
    \delta C^2  \left(1-\frac{p_{\Lambda | \Theta}(\lambda| \theta)}{\rho(\lambda)} \right)^2.
\end{align*}
Consequently, to show that $\frac{1}{2}\left\langle \nabla V_\delta(\vartheta),\bar g(\vartheta)\right\rangle\leq 0 $ for all $\vartheta$, it suffices to show that for all $\vartheta$
\begin{align}\label{eqn:diff}
    \left(\sum_{\tau \in \mathcal{L}} \bar g(\vartheta)(\tau)p_{\Lambda | \Theta}(\tau | \vartheta)\right)^{2}
 - \sum_{\tau\in \mathcal{L}}  [\bar g(\vartheta)(\tau)]^2 p_{\Lambda | \Theta}(\tau | \vartheta) +  \delta C^2  \left(1-\frac{p_{\Lambda | \Theta}(\lambda| \theta)}{\rho(\lambda)} \right)^2
\end{align}
is always non-positive.

To prove the last statement, recall that 
\[
    \bar g(\vartheta)(\tau) = 1-\frac{p_{\Lambda | \Theta}(\tau | \theta)}{\rho(\tau)}, \text{for } \tau \in \mathcal{L},
\]
and in order to make the discussion easier,  the following simplified notation will be used. Suppose that $\mathcal{L}=\{\lambda_1,\dots, \lambda_n\}$ for some $n\in\mathbb{N}$, denote $p_{\Lambda | \Theta}(\lambda_i | \theta)$ by $p_i$, $\rho(\lambda_i)$ by $\rho_i$, and let $c_i=p_i/\rho_i$ for all $i$. To show that \eqref{eqn:diff} is non-positive, it is equivalent to show that 
\[
    \left(\sum_{i=1}^n p_i(1-c_i)\right)^2 - \sum_{i=1}^n p_i(1-c_i)^2 + \delta C^2(1-c_k)^2 \leq 0,
\]
for some $k\in\{1,\dots,n\}$. Without loss of generality, assume $k=1$. 
Suppose $c_1=1$, i.e. $\rho_1= p_1$. Then, by Jensen's inequality,
\begin{align*}
    &\left(\sum_{i=1}^n p_i(1-c_i)\right)^2 - \sum_{i=1}^n p_i(1-c_i)^2 + \delta C^2(1-c_1)^2 \\
    &\qquad  \left(\sum_{i=1}^n p_i(1-c_i)\right)^2 - \sum_{i=1}^n p_i(1-c_i)^2 \leq 0.
\end{align*}
Suppose $c_1 \neq 1$. If $c_1 > 1$, then there must be some $j \neq 1$ such that $c_j<1$. Otherwise, $1 = \sum_{i=1}^n p_i > \sum_{i=1}^n \rho_i = 1$, which is a contradiction.  Similarly, if $c_1<1$, then there must be some $j\neq 1$ such that $c_j>1$ by the same argument. In either case, we may assume that $j = 2$, and it follows that $(c_1-c_2)^2>(1-c_1)^2$. With this inequality and the use of Lemma \ref{lem:jensen_diff}, it follows that
\begin{align*}
    &\left(\sum_{i=1}^n p_i(1-c_i)\right)^2 - \sum_{i=1}^n p_i(1-c_i)^2+\delta C^2(1-c_1)^2\\
    & \qquad = - \sum_{j=2}^n \frac{p_j\sum_{k=1}^{j-1} p_k}{\sum_{k=1}^j p_k}\left( (1-c_j)-\frac{\sum_{k=1}^{j-1} (1-c_k)p_k}{\sum_{k=1}^{j-1} p_k} \right)^2+\delta C^2(1-c_1)^2\\
    & \qquad \leq -\frac{p_2p_1}{p_1+p_2}((1-c_2)-(1-c_1))^2 +\delta C^2(1-c_1)^2\\
    & \qquad = -\frac{p_2p_1}{p_1+p_2}(c_1-c_2)^2 +\delta C^2(1-c_1)^2\\
    &\qquad \leq \left(-\frac{p_2p_1}{p_1+p_2}+\delta C^2\right)(1-c_1)^2\\
    &\qquad \leq (-m^2+\delta C^2)(1-c_1)^2\leq 0.
\end{align*}
Equality can only occur when $c_i$ are equal to $1$ for all $i$. This implies that $p_{\Lambda | \Theta}(\tau | \theta) = \rho(\tau)$ for each $\tau \in \mathcal{L}$, from which the conclusion $\theta =\theta_*$ follows. Consequently, the term in \eqref{eq:middle} is also zero and therefore $\frac{1}{2}\left\langle \nabla V_\delta(\vartheta),\bar g(\vartheta)\right\rangle
= 0$ if and only if $V_\delta(\vartheta) = 0$. 
This completes the proof.
\end{proof}

The following elementary result is used in the proof of Lemma \ref{lem_4_2}. 

\begin{lemma} \label{lem:jensen_diff}
For any integer $n\geq 2$, real numbers $\{x_i\}_{i=1}^n \subset\mathbb{R}$ and probability $\{p_i\}_{i=1}^n$ on $\{1, \dots, n\}$, 
\begin{align}\label{eqn:jensen_diff}
    \left(\sum_{i=1}^n x_i p_i\right)^2 - \sum_{i=1}^n x_i^2 p_i = - \sum_{j=2}^n \frac{p_j\sum_{k=1}^{j-1} p_k}{\sum_{k=1}^j p_k}\left( x_j-\frac{\sum_{k=1}^{j-1} x_kp_k}{\sum_{k=1}^{j-1} p_k} \right)^2.
\end{align}
\end{lemma}
\begin{proof}
When $n=2$, since $p_1+p_2=1$, it follows that the left-hand-side of \eqref{eqn:jensen_diff} is given by, 
\begin{align*}
    &(x_1p_1+x_2p_2)^2 - [x_1^2p_1+x_2^2p_2]\\
    &\quad=(x_1p_1+x_2(1-p_1))^2 - [x_1^2p_1+x_2^2(1-p_1)]\\
    &\quad = x_1^2p_1^2 +2x_1x_2p_1(1-p_1)+x_2^2(1-p_1)^2 - x_1^2p_1-x_2^2(1-p_1)\\
    &\quad = x_1^2p_1(p_1-1)+2x_1x_2p_1(1-p_1)-x_2^2p_1(1-p_1)\\
    &\quad = p_1(1-p_1)[-x_1^2+2x_1x_2-x_2^2] = -p_1(1-p_1)(x_1-x_2)^2.
\end{align*}
Additionally, the right-hand-side of \eqref{eqn:jensen_diff} equals
\[
     -  \frac{p_2 p_1}{\sum_{k=1}^2 p_k}\left( x_2-\frac{ x_1p_1}{ p_1} \right)^2 = -p_1p_2 (x_1-x_2)^2=-p_1(1-p_1)(x_1-x_2)^2.
\]
Thus, \eqref{eqn:jensen_diff} holds when $n=2$.

Now suppose that \eqref{eqn:jensen_diff} holds for some $n \geq 2$, then
\begin{align*}
    \left(\sum_{i=1}^{n+1} x_i p_i\right)^2 &= \left(x_{n+1}p_{n+1}+(1-p_{n+1})\sum_{i=1}^n x_i \frac{p_i}{1-p_{n+1}}\right)^2\\
    & = x_{n+1}^2 p_{n+1} + \left(\sum_{i=1}^n x_i \frac{p_i}{1-p_{n+1}}\right)^2 (1-p_{n+1})\\
    &\qquad - p_{n+1}(1-p_{n+1})\left(x_{n+1}-\sum_{i=1}^n x_i \frac{p_i}{1-p_{n+1}}\right)^2,
\end{align*}
where the second equality comes from applying \eqref{eqn:jensen_diff} with two terms. 

On the other hand, by assumption, \eqref{eqn:jensen_diff} holds for $n$,  which implies that 
\begin{align*}
    &\left(\sum_{i=1}^n x_i \frac{p_i}{1-p_{n+1}}\right)^2  \\&\qquad= \sum_{i=1}^n x_i^2 \frac{p_i}{1-p_{n+1}}  - \sum_{j=2}^n \frac{\frac{p_j}{1-p_{n+1}}\sum_{k=1}^{j-1} p_k}{\sum_{k=1}^j p_k}\left( x_j-\frac{\sum_{k=1}^{j-1} x_kp_k}{\sum_{k=1}^{j-1} p_k} \right)^2. 
\end{align*}
Consequently, 
\begin{align*}
    &\left(\sum_{i=1}^n x_i \frac{p_i}{1-p_{n+1}}\right)^2 (1-p_{n+1})\\
    &\qquad= \sum_{i=1}^n x_i^2 p_i  - \sum_{j=2}^n \frac{p_j\sum_{k=1}^{j-1} p_k}{\sum_{k=1}^j p_k}\left( x_j-\frac{\sum_{k=1}^{j-1} x_kp_k}{\sum_{k=1}^{j-1} p_k} \right)^2.
\end{align*}
Using the equality in the last display, we conclude that
\begin{align*}
    \left(\sum_{i=1}^{n+1} x_i p_i\right)^2 
    & = x_{n+1}^2 p_{n+1} + \sum_{i=1}^n x_i^2 p_i  - \sum_{j=2}^n \frac{p_j\sum_{k=1}^{j-1} p_k}{\sum_{k=1}^j p_k}\left( x_j-\frac{\sum_{k=1}^{j-1} x_kp_k}{\sum_{k=1}^{j-1} p_k} \right)^2\\
    &\qquad - p_{n+1}(1-p_{n+1})\left(x_{n+1}-\sum_{i=1}^n x_i \frac{p_i}{1-p_{n+1}}\right)^2\\
    & = \sum_{i=1}^{n+1} x_i^2 p_i  - \sum_{j=2}^{n+1} \frac{p_j\sum_{k=1}^{j-1} p_k}{\sum_{k=1}^j p_k}\left( x_j-\frac{\sum_{k=1}^{j-1} x_kp_k}{\sum_{k=1}^{j-1} p_k} \right)^2.
\end{align*}
The proof is complete by mathematical induction.
\end{proof}

\subsection{Convergence of ergodic averages}

In this subsection, the main theorem on the almost sure convergence of the estimates of the extended design parameters $\{\vartheta_n\}$ of the AWH algorithm is provided in Theorem \ref{thm:asconvergence_ext}. 
\begin{theorem}\label{thm:asconvergence_ext}
Let $\{\vartheta_n\}$ be given by  \eqref{eq:recursion_ext}, with target distribution $\rho$, $\rho(\lambda) > 0, \lambda \in \mathcal{L}$, $H$ a hyper-rectangle in $\mathbb{R}^{|\mathcal{L}_0|}$, and $V_\delta$ be given by \eqref{eq:Verg}, where $\delta > 0$ is sufficiently small as in Lemma \ref{lem_4_2}. Assume that $H$ is large enough so that $\{\vartheta\in H:V_\delta(\vartheta)=0\}$ is non-empty. Then  $\{\vartheta_n\}$ converges almost surely to $\{\vartheta \in H: V_\delta(\vartheta) = 0\}$. 
\end{theorem}

\begin{proof}
The proof of the statement in the theorem follows from almost identical arguments to those in the proof of Theorem \ref{thm:asconvergence}, with Lemma \ref{lem:gbar} and \ref{lem_4} replaced by Lemma \ref{lem:gbar2} and \ref{lem_4_2}, respectively. The main difference is that the upper bound \eqref{eq:hbound} of $h$ needs to be modified to the extended definition of $h$. 
To this end \eqref{eq:hbound} is replaced by
\begin{align*}
    |h(\vartheta,\xi)(u)| = \left\{\begin{array}{ll} 
    \left|1-\frac{\sum_{k=1}^{N_I} p_{\Lambda | X, \Theta}(u | \xi(k), \theta)}{N_I\rho(u)}\right|\leq 1+\frac{1}{\rho(u)}, & \text{ if } u \in \mathcal{L}, \\
    |\Phi(\xi, \theta) - \zeta|\leq 2 \max\limits_{x\in\mathcal{X}}|\phi(x)|, & \text{ if }  u =0. \end{array} \right.
\end{align*}
\end{proof}

\begin{corollary} \label{cor:fediff_ext}
Let $\{\vartheta_n\}$ be given by  \eqref{eq:recursion_ext}, with target distribution $\rho$, $\rho(\lambda) > 0, \lambda \in \mathcal{L}$, $H$ a hyper-rectangle in $\mathbb{R}^{|\mathcal{L}_0|}$, and $V_\delta$ be given by \eqref{eq:Verg}, where $\delta > 0$ is sufficiently small as in Lemma \ref{lem_4_2}. Assume that $H$ is large enough so that $\{\vartheta\in H:V_\delta(\vartheta)=0\}$ is non-empty. Then  for a fixed $\lambda \in \mathcal{L}$ and a function $\psi:\mathcal{X}\rightarrow\mathbb{R}$, the adaptive ergodic average
\begin{align*} 
    \frac{\sum_{n=0}^{N-1} \sum_{k=1}^{N_I}\psi(X_{nN_I + k})p_{\Lambda | X, \Theta}(\lambda | X_{nN_I + k}, \theta_n)}{\sum_{n=0}^{N-1}\sum_{k=1}^{N_I}p_{X | \Lambda, \Theta}(\lambda | X_{nN_I + k}, \theta_n)} 
\end{align*}
given in \eqref{eq:ergest} converges almost surely to 
\[
    \sum_{x\in\mathcal{X}}\psi(x)p_{X | \Lambda}(x | \lambda).
\] 
\end{corollary}

 \begin{proof}
 Recall that $\vartheta_n = (\theta_n,\zeta_n)$. By Theorem \ref{thm:asconvergence_ext} the sequence $\{\vartheta_n\}$ converges almost surely to the set $\{\vartheta \in H : V_\delta(\vartheta) = 0\}$. For any $\vartheta \in \mathbb{R}^{|\mathcal{L}|}\times[-C,C]$ with $V_\delta(\vartheta) = 0$ it follows that 
 \[
    \zeta = \sum_{x\in\mathcal{X}} \phi(x) p_{X,\Lambda | \Theta}(x,\lambda | \theta_*).
 \]
 Consequently, for any $\phi:\mathcal{X}\rightarrow\mathbb{R}$,
 \[
    \zeta_N= \hat \phi_N \doteq \frac{1}{N\,N_I}\sum_{n=0}^{N-1} \sum_{k=1}^{N_I}\phi(X_{nN_I + k})p_{\Lambda | X, \Theta}(\lambda | X_{nN_I + k}, \theta_n)
 \]
 converges almost surely to 
 \[
    \zeta = \sum_{x\in\mathcal{X}} \phi(x) p_{X,\Lambda | \Theta}(x,\lambda | \theta_*).
 \]
 Moreover, by taking $\phi=\psi$ and $\phi\equiv 1$, and since the ratio of two almost surely convergent sequences is still an almost surely convergent sequence, we know that
\begin{align*} 
    \frac{\sum_{n=0}^{N-1} \sum_{k=1}^{N_I}\psi(X_{nN_I + k})p_{\Lambda | X, \Theta}(\lambda | X_{nN_I + k}, \theta_n)}{\sum_{n=0}^{N-1}\sum_{k=1}^{N_I}p_{X | \Lambda, \Theta}(\lambda | X_{nN_I + k}, \theta_n)}  
    = \frac{\hat \psi_N}{\hat 1_N} 
\end{align*}
converges almost surely to 
\[
   \frac{\sum_{x\in\mathcal{X}} \psi(x) p_{X,\Lambda | \Theta}(x,\lambda | \theta_*)}{\sum_{x\in\mathcal{X}}  p_{X,\Lambda | \Theta}(x,\lambda | \theta_*)} = \sum_{x\in\mathcal{X}}\psi(x)p_{X | \Lambda}(x | \lambda).
\]
This completes the proof. 
\end{proof}

\appendix

\section{Convergence of Stochastic Approximations}
In order to make this paper self-contained, some notations and statements used in \cite{KY03} are introduced in this Appendix.  Additionally, Theorem 6.1 from \cite[Chapter 6]{KY03} is stated, which is the main result that will be applied to show the almost sure convergence of the AWH algorithm in the proof of Theorem \ref{thm:asconvergence} in this paper.

Consider a stochastic approximation with projection onto a constraint set $H$ defined as follows:
\begin{align}\label{eqn:SA}
    \theta_{n+1}=\theta_n + \varepsilon_n Y_n +\varepsilon_n Z_n,
\end{align}
where $Y_n$ is the observation at time $n$ and $Z_n$ is the projection term. Define $\mathcal{F}_n$ as the $\sigma$-algebra determined by the initial condition $\theta_0$ and observations $Y_i,i<n$, and assume that
there exists measurable functions $g_n(\cdot,\cdot)$ such that
\begin{align}\label{eqn:SA2}
    E_n Y_n = g_n(\theta_n,\xi_n)+\beta_n,
\end{align}
where $E_n$ denotes the expectation conditioned on $\mathcal{F}_n$, $\xi_n$ denotes the correlated driving noise, and $\beta_n$ represents a bias which is asymptotically unimportant.

Next, we define martingale differences $\delta M_n\doteq Y_n - E_nY_n$ and use \eqref{eqn:SA2} to write \eqref{eqn:SA} as 
\begin{align}\label{eqn:SA3}
    \theta_{n+1}=\theta_n + \varepsilon_n [ g_n(\theta_n,\xi_n)+\delta M_n+\beta_n +\varepsilon_n Z_n].
\end{align}
Moreover, define $t_0=0$ and $t_n=\sum_{i=0}^{n-1}\varepsilon_i$. For $t\geq 0$, let $m(t)$ denote the unique value of $n$ such that $t_n\leq t< t_{n+1}$.  

With these notations, we can now state Theorem 6.1 from \cite[Chapter 6]{KY03}. We only state the part that is relevant for the present paper.
\begin{theorem}\label{thm_1} Consider $\{\theta_n\}$ defined by \eqref{eqn:SA3}. Assume 
\begin{itemize}
    \item (5.1.1) 
        $\sum_{n=0}\varepsilon_n = \infty, \varepsilon_n\geq 0,\varepsilon_n\rightarrow 0, \mbox{ for } n\geq 0.$ 
    \item (A1.1) $\sup_n E|Y_n|<\infty.$
    \item (A1.2) $g_n(\theta,\xi)$ is continuous in $\theta$ for each $\xi$ and $n$.
    \item (A1.4) For each $\mu>0$ and some $T>0$,
    \[
        \lim_{n\rightarrow\infty}P\left\{ \sup_{j\geq n}\max_{0\leq t\leq T}\left|\sum_{i=m(jT)}^{m(jT+t)-1}\varepsilon_i\delta M_i\right|\geq \mu\right\} = 0.
    \]
    \item (6.2) $P(\xi_{n+1}\in\cdot|\xi_i,\theta_i,i\leq n) = p(\xi_n,\cdot|\theta_n)$, where $p(\xi,\cdot|\theta)$ denotes the one-step transition probability with starting point $\xi$ and parameterized by $\theta$.
    \item (A6.1) There is a continuous function $\bar{g}(\cdot)$ such that for $\theta\in H$, the expression
    \[
        v_n(\theta,\xi_n) = \sum_{i=n}^{\infty} \varepsilon_i E_n [g_i(\theta,\xi_i(\theta))-\bar{g}(\theta)]
    \]
    is well defined when the initial condition for $\{\xi_i,i\geq n\}$ is $\xi_n(\theta)=\xi_n$, and $v_n(\theta_n,\xi_n)\rightarrow 0$ w.p.1.
    \item (A6.2) The asymptotic rates of change of the processes
    \[
        A^0(t) = \sum_{i=0}^{m(t)-1}\alpha_i \mbox{ and }N^0(t) = \sum_{i=0}^{m(t)-1}\delta N_i 
    \]
    are zero with probability one, where $\alpha_n = v_{n+1}(\theta_{n+1},\xi_{n+1})-v_{n+1}(\theta_{n},\xi_{n+1})$ and $\delta N_n = v_{n+1}(\theta_{n},\xi_{n+1})-E_n v_{n+1}(\theta_{n},\xi_{n+1})$. In other words, for $\omega$ not in some null set and for any positive $T$,
    \[
        \lim_n\sup_{j\geq n}\max_{0\leq t\leq T} |A^0(\omega,jT+t)-A^0(\omega,jT)|=0.
    \]
    and 
    \[
        \lim_n\sup_{j\geq n}\max_{0\leq t\leq T} |N^0(\omega,jT+t)-N^0(\omega,jT)|=0.
    \]
\end{itemize}
and the following constraint set condition
\begin{itemize}
    \item (A4.3.1) $H$ is a hyperrectangle. In other words, there are real numbers $a_i<b_i, i=1,\dots,r$, such that $H=\{x:a_i\leq x_i\leq b_i\}$. 
\end{itemize}
Then for almost all $\omega$, $\{\theta_n(\omega)\}$ converges to some limit set of the ODE 
\[
    \dot{\theta}=\bar{g}(\theta) + z, z\in -C(\theta),
\] 
where $z$ is the minimum force needed to keep the solution in $H$.
\end{theorem}

\begin{remark}
According to the last paragraph on p.137, for some positive number $T$, 
\begin{align}\label{eqn_1}
    \lim_n\sup_{j\geq n}\max_{0\leq t\leq T}|M^0(jT+t)-M^0(jT)|=0 \text{ w.p.1.}
\end{align}
is guaranteed by (A2.1) $\sup_n E|Y_n|^2<\infty$ and (A2.4) $\sum_i \varepsilon_i^2<\infty$, where 
\[
    M^0(t) = \sum_{i=0}^{m(t)-1}\varepsilon_i \delta M_i, \quad \delta M_n = Y_n-E_n Y_n.
\]
Observe that \eqref{eqn_1} can be written as 
\[
     \lim_n\sup_{j\geq n}\max_{0\leq t\leq T}\left|\sum_{i=m(jT)}^{m(jT+t)-1}\varepsilon_i \delta M_i\right|=0 \text{ w.p.1},
\]
which implies (A1.4).
\end{remark}

\section*{Acknowledgements}

Financial support for Guo-Jhen Wu from the Brummer \& Partners MathDataLab and the Swedish e-Science Research Center (SeRC)  Data Science MCP is gratefully acknowledged. The authors would like to thank Jack Lidmar for introducing us to the accelerated weight histogram algorithm and numerous discussions on related topics. 

\pagebreak

\bibliographystyle{plain}


\end{document}